\documentclass[a4paper,11pt]{amsart}
\usepackage{amssymb,amsmath,amsthm}
\usepackage{graphicx,color}
\usepackage{stackrel}
\usepackage{hyperref}
\newcommand{\ie}{\emph{i.e.}}
\newcommand{\eg}{\emph{e.g.}}
\newcommand{\cf}{\emph{cf.}}

\newcommand{\Real}{\mathbb{R}}

\newcommand{\sii}{L^2}
\newcommand{\Dom}{\mathrm{Dom}}

\newcommand{\divergence}{\mathrm{div}}

\newcommand{\ds}{\displaystyle}
\newcommand{\essinf}{\mathop{\mathrm{ess\;\!inf}}}
\newcommand{\der}{\mathrm{d}}
\newtheorem{Lemma}{Lemma}
\newtheorem{Theorem}{Theorem}
\newtheorem{Corollary}{Corollary}
\newtheorem{Proposition}{Proposition}
\newtheorem{Conjecture}{Conjecture}
\theoremstyle{definition}
\newtheorem{Remark}{Remark}

\numberwithin{equation}{section}
\usepackage[normalem]{ulem}
\definecolor{DarkGreen}{rgb}{0,0.5,0.1} 

\newcommand\soutD{\bgroup\markoverwith
{\textcolor{DarkGreen}{\rule[.5ex]{2pt}{1pt}}}\ULon}
\newcommand\soutP{\bgroup\markoverwith
{\textcolor{blue}{\rule[.5ex]{2pt}{1pt}}}\ULon}
\newcommand{\Hm}[1]{\leavevmode{\marginpar{\tiny%
$\hbox to 0mm{\hspace*{-0.5mm}$\leftarrow$\hss}%
\vcenter{\vrule depth 0.1mm height 0.1mm width \the\marginparwidth}%
\hbox to
0mm{\hss$\rightarrow$\hspace*{-0.5mm}}$\\\relax\raggedright #1}}}

\begin{document}
%
\title[Extremal domains for Robin eigenvalues]{
Bounds and extremal domains for Robin eigenvalues with negative boundary parameter}
\author[Antunes, Freitas and Krej\v{c}i\v{r}\'{\i}k]{
Pedro R.~S. Antunes, Pedro Freitas and David Krej\v{c}i\v{r}\'{\i}k}

\address{Group of Mathematical Physics, Faculdade de Ci\^{e}ncias da Universidade de Lisboa, Campo Grande, Edif\'{\i}cio C6
1749-016 Lisboa, Portugal} \email{prantunes@fc.ul.pt}

\address{Department of Mathematics, Faculty of Human Kinetics \&
Group of Mathematical Physics, Faculdade de Ci\^{e}ncias da Universidade de Lisboa, Campo Grande, Edif\'{\i}cio C6
1749-016 Lisboa, Portugal}
\email{psfreitas@fc.ul.pt}

\address{
Department of Theoretical Physics,
Nuclear Physics Institute,
Academy of Sciences,
25068 \v{R}e\v{z}, Czech Republic
}
\email{krejcirik@ujf.cas.cz}

\date{May 26, 2015}

\thanks{This research was partially supported by FCT (Portugal) through project
PTDC/MAT-CAL/4334/2014. The first author was also supported through the program ``Investigador FCT''
with reference IF/00177/2013. The third author was also supported by
the project RVO61389005 and the GACR grant No.\ 14-06818S}

\begin{abstract}
We present some new bounds for the first Robin eigenvalue 
with a negative boundary parameter. These
include the constant volume problem, where the bounds are based on the shrinking coordinate method,
and a proof that in the fixed perimeter case the disk maximises the first eigenvalue for all values
of the parameter. 
This is in contrast with what happens in the constant area problem, where
the disk is the maximiser only for small values of the boundary parameter. We also present sharp
upper and lower bounds for the first eigenvalue of the ball 
and spherical shells.

These results are complemented by the numerical optimisation of 
the first four and two eigenvalues 
in~$2$ and~$3$ dimensions, respectively, 
and an evaluation of the quality of the upper bounds obtained. 
We also study the bifurcations from the ball as the boundary parameter 
becomes large (negative).

\bigskip
\begin{itemize}
\item[\textbf{Keywords:}]
eigenvalue optimisation,  
Robin Laplacian, negative boundary parameter,
Bareket's conjecture
\item[\textbf{MSC (2010):}]
58J50, 35P15
\end{itemize}
\end{abstract}

\maketitle
\section{Introduction}
%

The study of extremal eigenvalues of the Laplace operator which had its origin in Lord Rayleigh's
book {\it The Theory of Sound}~\cite{Rayleigh_1877} has by now been a continuous active topic of research
among mathematicians and physicists for nearly one century and a half. In the case of Dirichlet and Neumann
boundary conditions it has been known since the $1920$'s and the $1950$'s, respectively, that the ball optimises
the first eigenvalue among domains with fixed volume, being a minimiser in the first case and a maximiser
in the second~\cite{Faber_1923,Krahn_1924,Krahn_1926,Szego_1954,Weinberger_1956}.
In the case of Dirichlet boundary conditions this implies that the second eigenvalue is optimised by two equal
balls~\cite{Krahn_1926}, while for the Neumann problem it is only known that this provides a bound for planar
simply connected domains~\cite{Girouard-Nadirashvili-Polterovich_2009}.

Recent numerical work has shown that, with some rare exceptions,
notably that of the $(d+1)^\mathrm{th}$ Dirichlet eigenvalue
in dimension~$d$ which is conjectured to be optimised by the ball,
one cannot expect extremal domains in the
mid-frequency range to be defined explicitly in terms of known
functions~\cite{Antunes-Freitas_2012,Antunes-Freitas_2015,Oudet_2004}. For
planar domains with fixed area, 
it has been shown that the disk
cannot be a minimiser for eigenvalues higher
than the third~\cite{Berger_2015}. On the other hand, this has prompted the study of what happens in the high-frequency
limit as the order of the eigenvalue goes to infinity, where there is again some structure. In particular, it
was shown by Dorin Bucur and the second author of the present paper that for planar domains with fixed perimeter
extremal domains converge to the disk~\cite{Bucur-Freitas_2013}.
In the case of fixed measure, the first two authors of the
present paper showed that minimisers of the Dirichlet problem within the class of rectangles converge to the
square~\cite{Antunes-Freitas_2013} and, moreover, 
it has been shown that convergence to the disk
in the general case is
equivalent to the well-known P\'{o}lya conjecture~\cite{Colbois-El_Soufi_2014}.

In this paper we are interested in extremal domains for eigenvalues of the Robin Laplacian, that is,
\begin{equation}\label{eq:robin}
    	\left\{
    \begin{aligned}
        -\Delta u &= \lambda u &\quad &\text{in $\Omega$} \,,\\
        \frac{\partial u}{\partial\nu}+\alpha \;\! u&=0 & &
        \text{on $\partial\Omega$} \,,
    \end{aligned}
\right.
\end{equation}
where $\Omega$ is a bounded domain in $\Real^{d}$
with outer unit normal~$\nu$
and the boundary parameter $\alpha$ is a real constant.

In the case of fixed measure of~$\Omega$ and positive boundary parameter,
it was only in $1986$ that it was proven that the
disk is still the extremal domain in two dimensions~\cite{Bossel_1986}, while the extension of this
result to higher dimensions had to wait until $2006$~\cite{Daners_2006}. The corresponding result for the second
eigenvalue was obtained in~\cite{Kennedy_2009},
with two equal balls being the minimal domain, again matching the Dirichlet
result. However, due to the presence of a boundary parameter in the Robin problem, the behaviour for higher
eigenvalues will be, in principle, more complex as may be seen from the numerical results in~\cite{Antunes-Freitas-Kennedy_2013}.
In particular, for a given fixed volume~$|\Omega|$
and small positive values of the boundary parameter~$\alpha$
it is conjectured in that
paper that the $n^\mathrm{th}$ eigenvalue $\lambda_n^\alpha(\Omega)$
will in fact be minimised by~$n$ equal balls,
but that this will not be
the case for larger values of the parameter. This switching between
extremal domains as the parameter changes was recently shown by the second and third authors of the present
paper to also play a role for negative values of the parameter,
even in the case of the first eigenvalue~\cite{FK7}.
More precisely, while it was shown that in two dimensions the disk remains
an extremal domain for small (negative) values of the
parameter (now a maximiser),
thus proving the long standing {Bareket's conjecture~\cite{Bareket_1977}
in that case,
it was also shown that for larger
(negative) values of the parameter it cannot remain the optimiser.
This provides the first known example where the
extremal domain for the \emph{lowest} eigenvalue
of the Laplace operator is not a ball.

The proof that the disk cannot remain the optimiser for all values of the boundary parameter is based on the comparison
between the asymptotic behaviour of the eigenvalues of disks and annuli as the boundary parameter goes to minus infinity,
and carries over to higher dimensions. More precisely, while the asymptotic behaviour of a ball $B_R$ with radius $R$ in $\Real^{d}$ is
given by
\begin{equation}\label{as.ball}
  \lambda_1^\alpha(B_R) = -\alpha^2 + \frac{d-1}{R} \, \alpha + o(\alpha)
  \,,
  \qquad
  \alpha\to -\infty
  \,,
\end{equation}
that of a spherical shell 
$A_{R_1,R_2} :=B_{R_2}\setminus\overline{B_{R_1}}$
with  radii $R_1 < R_2$
is given by
\begin{equation}\label{as.shell}
  \lambda_1^\alpha(A_{R_1,R_2})
  = -\alpha^2 + \frac{d-1}{R_2} \, \alpha + o(\alpha)
  \,,
  \qquad
  \alpha\to -\infty
  \,.
\end{equation}
We thus see that, if $R_{2}$ is larger than $R$,
$\lambda_1^\alpha(A_{R_1,R_2})$ must become larger than $\lambda_1^\alpha(B_R)$
for sufficiently large negative $\alpha$.

On the other hand, let us recall that Bareket
has proved her conjecture already in~\cite{Bareket_1977}
for a class of ``nearly circular domains''
and Ferone, Nitsch and Trombetti~\cite{Ferone-Nitsch-Trombetti_2015}
have shown recently that it holds
within the ``class of Lipschitz sets which are `close'
to a ball in a Hausdorff metric sense''.
Our proof from~\cite{FK7} for small (negative) values of the parameter
differs from immediate results based on a simple perturbation argument
(see, \eg, \cite[Sec.~2.3]{Lacey-Ockendon-Sabina_1998})
in that the smallness of~$\alpha$ is shown to depend on
the area of~$\Omega$ only (\cf~\cite[Rem.~2]{FK7}).

In this paper we shall complete these results in the following directions.
We shall begin by providing a new upper bound
for the first eigenvalue under a fixed volume restriction.
\begin{Theorem}\label{Thm.bound}
Let $\alpha \leq 0$.
Let~$\Omega$ be a strictly star-shaped bounded domain in~$\Real^d$
with Lipschitz boundary~$\partial\Omega$ and let~$B$ be the ball
of the same volume.
Then
\begin{equation}\label{bound}
  \lambda_1^\alpha(\Omega)
  \leq \frac{F(\Omega)}{F(B)} \, \lambda_1^{\tilde{\alpha}}(B)
  \,, \qquad \mbox{where} \qquad
  \tilde\alpha := \alpha \, \frac{|\partial\Omega|}{|\partial B|} \,
  \frac{F(B)}{F(\Omega)}
  \,,
\end{equation}
where~$F(\Omega)$ is a geometric quantity
related to the support function of~$\Omega$
defined by~\eqref{Freitas}.
\end{Theorem}
\noindent
The proof of Theorem~\ref{Thm.bound}
is done using an approach based on {\it shrinking coordinates}
which were introduced in~\cite{Polya-Szego} in the two-dimensional
case and then extended to higher dimensions in~\cite{FK5}.

We then consider the case of fixed perimeter
in dimension two
for which we show that,
in contrast with the fixed area problem, the disk
is now the maximiser for all negative values
of the boundary parameter $\alpha$.
\begin{Theorem}\label{Thm.perimeter}
Let $\alpha \leq 0$.
For bounded planar domains $\Omega$ of class~$C^2$, we have
\[
 \lambda_{1}^\alpha(\Omega)\leq\lambda_{1}^\alpha(B),
\]
where $B$ is a disk with the same perimeter as $\Omega$.
\end{Theorem}
\noindent
The proof of Theorem~\ref{Thm.perimeter} is based on
an intermediate result from~\cite{FK7}
established with help of \emph{parallel coordinates}.

Complementing the above results, 
we provide sharp bounds for the first eigenvalue of 
the ball in any dimensions.
\begin{Theorem}\label{Thm.ballbounds}
Let $B_{R}$ denote the $d$-dimensional ball of radius~$R$
and denote by $\lambda_{1}^\alpha(B_{R})$ its first
Robin eigenvalue. Then we have
\[
 -\frac{\ds 1}{\ds 2}\alpha^{2}+\frac{\ds (d-1)}{\ds 2R}\alpha+ \frac{\ds \alpha}{\ds 2R}
 \sqrt{(d-1-\alpha R)^2+4d} <\lambda_{1}^\alpha(B_{R})< -\alpha^{2}+\frac{\ds (d-1)}{\ds R}\alpha
\]
for all negative $\alpha$.
\end{Theorem}
\noindent
Note that this result actually states that the first eigenvalue of the ball is smaller than the first two
terms in the asymptotic expansion~\eqref{as.ball} for all negative values of the boundary parameter $\alpha$.
Furthermore, it is not difficult to check that the lower bound satisfies the same two-term asymptotics, with the
next term being of order $O(1)$.
We also see that, for fixed $\alpha$, $\lambda_{1}^\alpha(B_{R})$ goes to $-\infty$
with $R^{-1}$ as $R$ goes to zero.
On the other hand, it follows from the upper bound that, 
whenever~$\alpha$ is negative, $\lambda_{1}^\alpha(B_{R})$
does not go to $0 = \lambda_{1}^\alpha(\Real^d)$ 
as $R \to \infty$ (see also Remark~\ref{Rem.nrs} below).

We also establish sharp lower and upper bounds for
the first Robin eigenvalue in $d$-dimensional 
spherical shells, which, as explained above, we conjecture are the extremal sets for large negative $\alpha$.
\begin{Theorem}\label{Thm.shell}
Let~$A_{R_1,R_2}$ denote the $d$-dimensional spherical shell 
with inner and outer radii 
given by~$R_1$ and~$R_2$, respectively,
and denote by $\lambda_1^\alpha(A_{R_1,R_2})$ its
first Robin eigenvalue. Then we have
\begin{equation}\label{bound.shell}
  - \alpha^2 + \left( \frac{d-1}{R_2} + \frac{2}{R_2-R_1} \right) \alpha
  < \lambda_1^\alpha(A_{R_1,R_2}) <
  - \alpha^2 + \frac{d-1}{R_2} \, \alpha
  \,.
\end{equation}
for all negative $\alpha$.
\end{Theorem}
\noindent
Here the upper bound is optimal up to the second order as $\alpha \to -\infty$,
\cf~\eqref{as.shell}, and we see that, as in the case of the ball, the eigenvalue is bounded by the first two
terms in the asymptotics. The lower bound follows the first-order
asymptotics only.

Finally, we perform a numerical study to obtain insight into the structure that is to be expected for this
problem. In particular, our results support the conjecture that the first eigenvalue is also maximised by the ball for small
negative values of $\alpha$ in three dimensions 
and that both in this case and in the plane shells with varying (increasing)
radii become the extremal domain as $\alpha$ becomes more negative.
We also study the optimisation of higher eigenvalues.
Based on these numerical simulations, we formulate some conjectures at the end of the paper.

\section{The first eigenvalue of balls and shells}
%
Let~$\Omega$ be a bounded domain in~$\Real^d$ ($d \geq 2$)
with Lipschitz boundary~$\partial\Omega$ and $\alpha \in \Real$.
As usual, we understand~\eqref{eq:robin} as a spectral problem for
the self-adjoint operator $-\Delta_\alpha^\Omega$
in the Hilbert space $\sii(\Omega)$
associated with the closed quadratic form
\begin{equation}\label{form}
  Q_\alpha^\Omega[u] := \|\nabla u\|_{\sii(\Omega)}^2
  +\alpha\;\!\|u\|_{\sii(\partial\Omega)}^2
  \,, \qquad
  \Dom(Q_\alpha^\Omega) := W^{1,2}(\Omega)
  \,.
\end{equation}
The lowest point in the spectrum of $-\Delta_\alpha^\Omega$
can be characterised by the variational formula
\begin{equation}\label{Rayleigh}
  \lambda_1^\alpha(\Omega)
  = \inf_{\stackrel[u\not=0]{}{u \in W^{1,2}(\Omega)}}
  \frac{Q_\alpha^\Omega[u]}
  {\,\|u\|_{\sii(\Omega)}^2}
  \,.
\end{equation}
Since the embedding $W^{1,2}(\Omega) \hookrightarrow \sii(\Omega)$ is compact,
we know that $\lambda_1^\alpha(\Omega)$ is indeed a discrete eigenvalue
and the infimum is achieved by a function $u_1^\alpha \in W^{1,2}(\Omega)$.

Using a constant test function in~\eqref{Rayleigh},
we get
\begin{equation}\label{bound.trivial}
  \lambda_1^\alpha(\Omega) \leq
  \alpha \, \frac{|\partial\Omega|}{|\Omega|}
  \,.
\end{equation}
Here $|\cdot|$ denotes the $d$-dimensional Lebesgue measure in the denominator
and the $(d-1)$-dimensional Hausdorff measure in the numerator.
It follows that $\lambda_1^\alpha(\Omega)$ is negative whenever $\alpha < 0$.

Now let~$B_R$ be a $d$-dimensional ball of radius~$R$.
By the rotational symmetry and regularity,
we deduce from~\eqref{Rayleigh}
\begin{equation}\label{Rayleigh.ball}
  \lambda_1^\alpha(B_R)
  = \inf_{\stackrel[\phi\not=0]{}{\phi \in C^1([0,R])}}
  \frac{\displaystyle \int_0^R |\psi'(r)|^2 \, r^{d-1} \, \der r
  + \alpha \, R^{d-1} \, |\phi(R)|^2}
  {\displaystyle \int_0^R |\phi(r)|^2 \, r^{d-1} \, \der r}
  \,.
\end{equation}
We know that $\lambda_1^\alpha(B_R)$ is simple
and that the infimum in~\eqref{Rayleigh.ball}
is achieved by a smooth positive function~$\phi_1$ satisfying
\begin{equation}\label{elliptic}
\left\{
\begin{aligned}
  -r^{-(d-1)}[r^{d-1} \phi'(r)]' &= \lambda \;\! \phi(r) \,,
  && r \in [0,R] \,,
  \\
  \phi'(0) &= 0 \,,
  \\
  \phi'(R) + \alpha \;\! \phi(R) &=0 \,.
\end{aligned}
\right.
\end{equation}

In fact, if $\alpha \leq 0$, we have an explicit solution
\begin{equation}\label{explicit}
  \phi_1(r) = r^{-\mu} \, I_{\mu}(kr)
  \,, \qquad
  \mu := \frac{d-2}{2}
  \,,
\end{equation}
where~$I_\mu$ is a modified Bessel function \cite[Sec.~9.6]{Abramowitz-Stegun}
and $k := \sqrt{-\lambda_1^\alpha(B_R)}$
is the smallest non-negative root of the equation
\begin{equation}\label{implicit}
  k I_\mu'(kR)
  - \mbox{$\frac{\mu}{R}$} I_\mu(kR)
  + \alpha I_\mu(kR) = 0
  \,.
\end{equation}
Using the identity (\cf~\cite[Sec.~9.6.26]{Abramowitz-Stegun})
\begin{equation}\label{AS.identity}
  I_\mu'(z) = I_{\mu+1}(z) + \mbox{$\frac{\mu}{z}$} I_\mu(z)
  \,,
\end{equation}
we see that~\eqref{implicit} is equivalent to
\begin{equation}\label{implicit.bis}
  k I_{\mu+1}(kR) + \alpha I_\mu(kR) = 0
  \,.
\end{equation}
\begin{Lemma}\label{Lem.monotonicity}
Let $\alpha < 0$. We have
\begin{equation}
  \forall r \in (0,R) \,, \qquad
  k r \, I_{\mu+1}(kr) + \alpha R \, I_\mu(kr) < 0
  \,,
\end{equation}
where~$k$ is the smallest positive root of~\eqref{implicit.bis}.
\end{Lemma}
\begin{proof}
If $\alpha < 0$, $k$~is the smallest \emph{positive} root of~\eqref{implicit.bis}.
Using the asymptotic formulae for small values of arguments
of Bessel functions (\cf~\cite[Sec.~9.6.7]{Abramowitz-Stegun}),
we know that
$
  z I_{\mu+1}(z) + \alpha R I_\mu(z)
$
is negative for all sufficiently small positive~$z$,
and hence for all~$z$ less than~$kR$,
where~$k$ is the first positive root of~\eqref{implicit.bis}.
\end{proof}

\subsection{An upper bound for \texorpdfstring{$\lambda_{1}^{\alpha}(B_{R})$}{ev}}
Now we give a proof of the upper bound in Theorem~\ref{Thm.ballbounds}.
Choosing in~\eqref{Rayleigh.ball} the test function

\begin{equation}\label{test}
  \phi(r) := e^{\alpha (R-r)}
  \,,
\end{equation}
we obtain the bound
\begin{equation}\label{pre.Gamma}
  \lambda_1^\alpha(B_R)
  \leq \alpha^2
  - \frac{2\alpha^2}{\Gamma_{d}(2\alpha R)}
\end{equation}
with
\begin{equation}\label{Gamma}
  \Gamma_{d}(x) := \int_x^0 \left(\frac{t}{x}\right)^{d-1} e^{-t+x} \, \der t
  \,.
\end{equation}
Hence the upper bound of Theorem~\ref{Thm.ballbounds}
follows provided that (recall that~$\alpha$ is negative)
\begin{equation}\label{crucial}
  \forall x < 0 \,, \qquad
  f(x) := \frac{x}{d-1} \frac{\Gamma_d(x)-1}{\Gamma_d(x)}
  > 1
  \,.
\end{equation}
To prove~\eqref{crucial}, we first notice the identity
\begin{equation}\label{first}
  \Gamma_d(x) = 1 + \frac{d-1}{x} \, \Gamma_{d-1}(x)
  \,,
\end{equation}
which can be established by an integration of parts.
Second, we have
\begin{equation}\label{second}
  \Gamma_{d-1}(x)
  = \int_x^0 \left(\frac{t}{x}\right)^{d-1} \frac{x}{t} \, e^{-t+x} \, \der t
  > \Gamma_{d}(x)
  \,,
\end{equation}
because $x/t > 1$ for all $t \in (x,0)$ with $x < 0$.
As a consequence of~\eqref{first} and~\eqref{second}, we have
\begin{equation}
  f(x) = \frac{\Gamma_{d-1}(x)}{\Gamma_d(x)}
  > 1
  \,,
\end{equation}
which proves~\eqref{crucial}
and thus concludes the proof of the upper bound of Theorem~\ref{Thm.ballbounds}.
\hfill\qed

\subsection{A lower bound for \texorpdfstring{$\lambda_{1}^{\alpha}(B_{R})$}{ev}}
To obtain the lower bound in Theorem~\ref{Thm.ballbounds} we shall use a different strategy
based directly on equation~\eqref{implicit.bis} and properties of quotients of Bessel functions.
A similar approach may also be used as an alternative way of establishing the upper bound in the previous section.

From~\cite{Amos_1974} we have
\begin{equation}
 \label{amosquotient}
 p_{\mu}(z):= \frac{\ds I_{\mu+1}(z)}{\ds I_{\mu}(z)}
  > \frac{\ds z}{\ds (\mu+1/2)+\sqrt{z^2+(\mu+3/2)^2}} \,,
  \;\; \mu,z>0 \,,
\end{equation}
where justification that strict inequality holds may be found in~\cite{Segura_2011}.
Applying this to equation~\eqref{implicit.bis} for negative $\alpha$ yields
\[
 k=-\alpha \frac{\ds I_{\mu}(k R)}{\ds I_{\mu+1}(k R)}
 = -\frac{\ds \alpha}{\ds p_{\mu}(k R)}< -\frac{\ds \alpha}{\ds kR}
 \left[ \mu+\frac{\ds 1}{\ds 2} + \sqrt{k^2R^2+\left(\mu+\frac{\ds 3}{\ds 2}\right)^2}\right],
\]
from which it follows that
\begin{equation}\label{aux1}
  k^2R+\alpha\left(\mu+\frac{\ds 1}{\ds 2}\right)
  <-\alpha \, \sqrt{k^2R^2+\left(\mu+\frac{\ds 3}{\ds 2}\right)^2} \,.
\end{equation}
Due to the upper bound proved in the previous section we know that
\[
 -k^2 < -\alpha^2+(2\mu+1)\frac{\ds \alpha}{\ds R}
\]
and thus
\[
  k^2R+\left(\mu+\frac{\ds 1}{\ds 2}\right)\alpha
  > \alpha^2 R-\left(\mu+\frac{\ds 1}{\ds 2}\right)\alpha>0
  \,,
\]
from which it follows that the left-hand side of~\eqref{aux1} is positive.
We may thus square both sides of~\eqref{aux1}
to obtain
\[
 k^4R^2+\alpha R\left[2\left(\mu+\frac{\ds 1}{\ds 2}\right)
 -\alpha R\right]k^2-2\left( \mu+1 \right)\alpha^2<0.
\]
This implies both a lower and an upper bounds for $k^2$
and while the lower bound is trivial, the upper bound yields the
desired lower bound for the eigenvalue $\lambda_{1}^{\alpha}(B_{R})=-k^2$.

\subsection{Monotonicity for balls}\label{Sec.monotonicity}
In this subsection we give a proof of the following monotonicity result.
\begin{Theorem}\label{Thm.monotonicity}
Let~$B_R$ be a ball of radius~$R$.
If $\alpha < 0$, then
$$
  R \mapsto \lambda_1^\alpha(B_R)
  \quad \mbox{is strictly increasing.}
$$
\end{Theorem}

We remark that a \emph{non-strict} monotonicity of
the first Robin eigenvalue for certain domains
(including balls) has been obtained in~\cite{Giorgi-Smits_2005}
(see also \cite{Giorgi-Smits_2007} and \cite{Giorgi-Smits_2008}).
Since our proof of Theorem~\ref{Thm.monotonicity} employs
different ideas and yields the \emph{strict} monotonicity,
we have decided to present it here.

For simplicity, throughout this subsection
we set $\lambda_R := \lambda_1^\alpha(B_{R})$.
We also write $\phi_R := \phi_1$ to stress the dependence
of the eigenfunction on the radius.

Our proof of Theorem~\ref{Thm.monotonicity} is based
on the following formula for the derivative of~$\lambda_R$ with respect to~$R$.
\begin{Lemma}
We have
\begin{equation}\label{derivative}
  \frac{\partial \lambda_R}{\partial R}
  = \frac{\displaystyle - \frac{2}{R} \int_0^R \phi_R'(r)^2 \,r^{d-1} \, \der r
  - \alpha \, R^{d-2} \, \phi_R(R)^2}
  {\displaystyle \int_0^R \phi_R(r)^2 \,r^{d-1} \, \der r}
  \,.
\end{equation}
\end{Lemma}
\begin{proof}
By the unitary transform $(U\phi)(\rho) := R^{d/2} \phi(R\rho)$,
we see that~$\lambda_R$ is the lowest eigenvalue of the operator~$T_R$
in the $R$-independent Hilbert space $\sii((0,1),\rho^{d-1} \der \rho)$
associated with the quadratic form
$$
\begin{aligned}
  t_R[f] &:= \frac{1}{R^2} \int_0^1 f'(\rho)^2 \,\rho^{d-1} \, \der\rho
  + \frac{1}{R} \, \alpha \, f(1)^2
  \,,
  \\
  \Dom(t_R) &:= W^{1,2}((0,1),\rho^{d-1} \der\rho)
  \,.
\end{aligned}
$$
Since~$T_R$ forms a holomorphic family in~$R$
(\cf~\cite[Thm.~VII.4.8]{Kato}) and~$\lambda_R$ is simple,
$\lambda_R$ and the associated eigenprojection
are holomorphic functions of~$R$.
In particular, the eigenfunction~$f_R$ of~$T_R$
associated with~$\lambda_R$
can be chosen to depend continuously on~$R$ in the topology of
$W^{1,2}((0,1),\rho^{d-1} \der\rho)$.
Now, the weak formulation of the eigenvalue problem for~$T_R$ reads
\begin{equation}\label{weak}
  t_R(\varphi,f_R) = \lambda_R \, (\varphi,f_R)
\end{equation}
for every $\varphi \in W^{1,2}((0,1),\rho^{d-1} \der\rho)$,
where $t_R(\cdot,\cdot)$ and $(\cdot,\cdot)$ denote
the sesqui\-linear form associated with $t_R[\cdot]$
and the inner product in $\sii((0,1),\rho^{d-1} \der\rho)$, respectively.
Differentiating~\eqref{weak} with respect to~$R$
(which is justified by the holomorphic properties of~$\lambda_R$ and~$f_R$)
and employing~\eqref{weak} in the resulting identity,
we conclude with
$$
  \frac{\partial \lambda_R}{\partial R}
  =
  \frac{\displaystyle - \frac{2}{R^3} \int_0^1 f_R'(r)^2 \,r^{d-1} \der r
  - \frac{1}{R^2} \, \alpha \, f_R(1)^2}
  {\displaystyle \int_0^1 f_R(r)^2 \,r^{d-1} \der r}
  \,.
$$
This formula coincides with~\eqref{derivative}
through the unitary identification~$U$.
\end{proof}

The derivative~\eqref{derivative} is clearly negative
whenever~$\alpha$ is positive.
At the same time, the derivative~\eqref{derivative}
is zero for the Neumann case $\alpha=0$.
If~$\alpha$ is negative, the numerator of~\eqref{derivative}
consists of a negative and a positive term,
so the sign of the derivative is not obvious in this case.
Our strategy to prove Theorem~\ref{Thm.monotonicity}
is to show that the derivative is in fact positive
whenever~$\alpha$ is negative.

By employing~\eqref{Rayleigh.ball},
where the infimum is actually achieved for~$\phi_R$,
we can rewrite~\eqref{derivative} as follows
\begin{equation}\label{derivative.bis}
  \frac{\partial \lambda_R}{\partial R}
  = \frac{\displaystyle - \frac{2}{R} \, \lambda_R \int_0^R \phi_R(r)^2 \,r^{d-1} \der r
  + \alpha \, R^{d-2} \, \phi_R(R)^2}
  {\displaystyle \int_0^R \phi_R(r)^2 \,r^{d-1} \der r}
  \,.
\end{equation}
Note that the sign of the numerator is still unclear
because $\lambda_R < 0$ whenever $\alpha < 0$, \cf~\eqref{bound.trivial}.
However, expression~\eqref{derivative.bis} is convenient
because of the following lemma.	
\begin{Lemma}
Let $\alpha < 0$. Then
\begin{equation}\label{explicit.bound}
  \int_{0}^{R} \phi_R(r)^2 \, r^{d-1} \, \der r
  > \frac{R^d}{2} \,
  \frac{\phi_R(R)^2}{1+\mu-\alpha R}
  \,.
\end{equation}
\end{Lemma}
\begin{proof}
Using~\eqref{explicit} and integrating by parts, we have
$$
\begin{aligned}
  \|\phi\|^2 := \int_{0}^{R} \phi_R(r)^2 \, r^{d-1} \, \der r
  &= \int_{0}^{R} I_\mu(kr)^2 \, r \, \der r
  \\
  &= I_\mu(kR)^2 \frac{R^2}{2}
  - k \int_{0}^{R} I_\mu(kr) \, I_\mu'(kr) \, r^2 \, \der r
  \,.
\end{aligned}
$$
Using the identity~\eqref{AS.identity},
we obtain
$$
  (1+\mu) \|\phi\|^2
  = I_\mu(kR)^2 \frac{R^2}{2}
  - k \int_{0}^{R} I_\mu(kr) \, I_{\mu+1}(kr) \, r^2 \, \der r
  \,.
$$
By Lemma~\ref{Lem.monotonicity}, it follows
$$
  (1+\mu) \|\phi\|^2
  > I_\mu(kR)^2 \frac{R^2}{2}
  + \alpha R \, \|\phi\|^2
  = \phi_R(R)^2 \frac{R^d}{2}
  + \alpha R \, \|\phi\|^2
  \,,
$$
which gives the desired claim.
\end{proof}

Assuming now that~$\alpha$ is negative
and using~\eqref{explicit.bound} in~\eqref{derivative.bis},
we thus obtain
\begin{equation}
  \frac{\partial \lambda_R}{\partial R}
  > \frac{\displaystyle
  R^{d-1} \, \phi_R(R)^2
  \left[-\lambda_R - \alpha^2 + (1+\mu) \;\! \frac{\alpha}{R} \right]
  }
  {\displaystyle (1+\mu-\alpha R) \int_0^R \phi_R(r)^2 \,r \, \der r}
  \,.
\end{equation}
Here the right-hand side is positive because
$$
  -\lambda_R - \alpha^2 + (1+\mu) \;\! \frac{\alpha}{R}
  > -\mu \;\! \frac{\alpha}{R}
  \geq 0
  \,,
$$
where the first inequality employs the upper bound of Theorem~\ref{Thm.ballbounds}
and the dependence of~$\mu$ on~$d$, \cf~\eqref{explicit}.
This concludes the proof of Theorem~\ref{Thm.monotonicity}.

\subsection{Bounds for \texorpdfstring{$\lambda_1^\alpha(A_{R_1,R_2})$}{ev1}}
In this subsection, we establish Theorem~\ref{Thm.shell}
dealing with $d$-dimensional spherical shells.
\begin{proof}[Proof of Theorem~\ref{Thm.shell}]
Now we have the variational characterisation
\begin{equation}\label{Rayleigh.shell}
  \lambda_1^\alpha(A_{R_1,R_2})
  = \inf_{\stackrel[\phi\not=0]{}{\phi \in C^1([R_1,R_2])}}
  \frac{\displaystyle \int_{R_1}^{R_2}|\phi'(r)|^2 \, r^{d-1} \, \der r 
  + \alpha \, R_1^{d-1} \, |\phi(R_1)|^2
  + \alpha \, R_2^{d-1} \, |\phi(R_2)|^2}
  {\displaystyle \int_{R_1}^{R_2} |\phi(r)|^2 \, r^{d-1} \, \der r}
  \,.
\end{equation}

To prove the lower bound, 
let~$\phi$ be a positive minimiser of~\eqref{Rayleigh.shell}.
Given any Lipschitz-continuous function $\eta:[R_1,R_2]\to\Real$ 
such that
\begin{equation*}
  \eta(R_2)=1
  \qquad\mbox{and}\qquad
  \eta(R_1) = -1
  \,,
\end{equation*}
we write
\begin{eqnarray*}
  \lefteqn{
  R_1^{d-1} \, \phi(R_1)^2 + R_2^{d-1} \, \phi(R_2)^2
  = \int_{R_1}^{R_2} [r^{d-1} \phi(r)^2 \eta(r)]' \,\der r
  }
  \\
  &&= \int_{R_1}^{R_2} \left[
  2 \phi(r) \phi'(r) \eta(r) + \phi(r)^2 \eta'(r) + (d-1) \phi(r)^2 \frac{\eta(r)}{r}
  \right] r^{d-1}\,\der r
  \,.
\end{eqnarray*}
Denoting by $Q[\phi]$ the numerator 
of the right-hand side of~\eqref{Rayleigh.shell},
we therefore have
\begin{equation}\label{strict}
\begin{aligned}
  Q[\phi] 
  &= \int_0^R 
  \left\{
  \big[\phi'(r) +\alpha \eta(r)\phi(r)\big]^2
  - \alpha^2 \eta(r)^2 + \alpha \eta'(r) + \alpha (d-1) \frac{\eta(r)}{r}
  \right\} \phi(r)^2 \, r^{d-1}\,\der r
  \\
  &\geq 
  \int_0^R 
  \left\{
  - \alpha^2 \eta(r)^2 + \alpha \eta'(r) + \alpha (d-1) \frac{\eta(r)}{r}
  \right\} \phi(r)^2 \, r^{d-1}\,\der r
  \\
  &\geq 
  \mu \int_0^R \phi(r)^2 \, r^{d-1}\,\der r
  \,, 
\end{aligned}
\end{equation}
where
\begin{equation}\label{mu}
  \mu := \sup_\eta \inf_{r \in (R_1,R_2)} f_\eta(r)
\end{equation}
with
$$
  f_\eta(r):=- \alpha^2 \eta(r)^2 + \alpha \eta'(r) + \alpha (d-1) \frac{\eta(r)}{r}
$$
We remark that the second inequality in~\eqref{strict} is strict
provided that the function~$f_\eta$ is not constant. 
In this case, we deduce from~\eqref{strict}
\begin{equation}\label{deduce}
  \lambda_1^\alpha(A_{R_1,R_2}) > \mu 
  \,.
\end{equation}
The function~$\eta$ can be understood as a sort of test function.

Choosing now
\begin{equation}\label{special}
  \eta(r) := \frac{2r - (R_1+R_2)}{R_2-R_1}
  \,,
\end{equation}
we obtain
\begin{equation*}
  f_{\eta}(r) = -\alpha^2 \left(\frac{2r-(R_1+R_2)}{R_2-R_1}\right)^2
  + \frac{2\alpha}{R_2-R_1}
  + \frac{\alpha(d-1)}{R_2-R_1} \left(2-\frac{R_1+R_2}{r}\right) 
  \,.
\end{equation*}
The function $r \mapsto f_{\eta}(r)$ 
is clearly non-constant and we claim that it is minimised at $r = R_2$.
To see the latter, we write
\begin{eqnarray*}
  \lefteqn{
  f_{\eta}(r) - f_\eta(R_2) 
  }
  \\
  &&= -\alpha^2 \left[\left(\frac{2r-(R_1+R_2)}{R_2-R_1}\right)^2-1\right]
  + \alpha(d-1)
  \left[\frac{2r-(R_1+R_2)}{(R_2-R_1) r}- \frac{1}{R_2}\right]
  \\
  &&= -\alpha^2 \frac{4(r-R_2)(r-R_1)}{(R_2-R_1)^2}
  + \alpha(d-1)
  \frac{(R_1+R_2)(r-R_2)}{R_2(R_2-R_1) r}
  \geq 0
  \,,
\end{eqnarray*}
where the last two terms are individually non-negative 
because $r \in [R_1,R_2]$ and~$\alpha$ is negative.
Consequently, with the special choice~\eqref{special},
\begin{equation}
  \mu = f_\eta(R_2) 
  = - \alpha^2 + \frac{2\alpha}{R_2-R_1} 
  + \frac{\alpha(d-1)}{R_2} 
  \,.
\end{equation}
This proves the lower bound of Theorem~\ref{Thm.shell} 
due to~\eqref{deduce}.

We now turn to the upper bound. 
In the variational characterisation~\eqref{Rayleigh.shell},
we again use the test function~\eqref{test} 
with~$R$ being replaced by~$R_2$ now.
It leads to the upper bound
\begin{equation} 
  \lambda_1^\alpha(A_{R_1,R_2})
  \leq \alpha^2 
  - \frac{\displaystyle 2\alpha^2 
  \left[1+\left(\frac{R_1}{R_2}\right)^{d-1} e^{2\alpha(R_2-R_1)} \right]}
  {\displaystyle \Gamma_{d}(2\alpha R_2,2\alpha R_1)}
\end{equation}
with
\begin{equation} 
  \Gamma_{d}(x,y) := \int_x^y \left(\frac{t}{x}\right)^{d-1} e^{-t+x} \, \der t
  \,.
\end{equation}
Since 
$
  0
  < \Gamma_{d}(2\alpha R_2,2\alpha R_1)
  < \Gamma_{d}(2\alpha R_2,0)
  \equiv \Gamma_{d}(2\alpha R_2)
$,
where~$\Gamma_{d}(2\alpha R_2)$ is defined in~\eqref{Gamma},
we have
\begin{equation} 
  \lambda_1^\alpha(A_{R_1,R_2})
  < \alpha^2 
  - \frac{\displaystyle 2\alpha^2 
  \left[1+\left(\frac{R_1}{R_2}\right)^{d-1} e^{2\alpha(R_2-R_1)} \right]}
  {\Gamma_{d}(2\alpha R_2)}
  < \alpha^2 
  - \frac{2\alpha^2}
  {\displaystyle \Gamma_{d}(2\alpha R_2)}
  \,.
\end{equation}
Here the last bound coincides with the upper bound~\eqref{pre.Gamma} for balls.
Using~\eqref{crucial}, we thus obtain the upper bound of Theorem~\ref{Thm.shell}.
\end{proof}
\begin{Remark}
Adapting the idea of the lower-part proof above to balls,
we obtain $\lambda_1^\alpha(B_R) > \mu$, 
where~$\mu$ is defined as in~\eqref{mu}
with $R_1 := 0$ and $R_2 := R$  
and $\eta:[0,R]\to\Real$ is any Lipschitz-continuous (test) function
such that 
\begin{equation*}
  \eta(R)=1
  \qquad\mbox{and}\qquad
  r^{d-1} \eta(r) \big|_{r=0} = 0
  \,.
\end{equation*}
Taking $\eta(r) := r/R$, we then arrive at the lower bound
\begin{equation}\label{bound.ball}
  - \alpha^2 + \alpha \, \frac{d}{R} 
  < \lambda_1^\alpha(B_R)
  \,.
\end{equation}
This is not as good as the lower bound of Theorem~\ref{Thm.ballbounds},
but, on the other hand, it is much simpler.
\end{Remark}

\subsection{The Robin problem in dimension \texorpdfstring{$d=1$}{1D}}
Finally, let us make a few comments on problem~\eqref{eq:robin}
when $\Omega$ is an interval. While t}he one-dimensional situation 
is formally excluded from this paper,
our main results still hold in that case. 
In fact, the case of $d=1$ is simpler in the sense 
that it can be reduced to a single transcendental equation.

To be more specific, without loss of generality,
let us assume now that~$\Omega$ is a one-dimensional ball
of radius $R>0$ centred at the origin, 
\ie\ $\Omega := B_R = (-R,R)$.
From~\eqref{bound.trivial} we know that 
$\lambda_1^\alpha(\Omega)$ is negative as long as $\alpha < 0$.
By solving the differential equation in~\eqref{eq:robin}
in terms of exponentials, subjecting the general solution
to the boundary conditions at~$\pm R$
and employing the fact that the corresponding eigenfunction 
cannot change sign, 
we obtain that
$\lambda_1^\alpha((-R,R)) = -k^2$, 
where~$k$ is the smallest positive solution of
\begin{equation}\label{trans}
  - \frac{k}{\alpha} = \coth(R k)
  \,.
\end{equation}

First of all, we study the asymptotic regime
when the boundary parameter goes to minus infinity.
\begin{Proposition}\label{Prop.as.1D}
We have the asymptotics
\begin{equation}\label{as.1D}
  \lambda_1^\alpha((-R,R))
  = - \alpha^2 - 4 \alpha^2 e^{2 R \alpha} + o(\alpha^2 e^{2 R \alpha})  
  \,, \qquad 
  \alpha \to -\infty
  \,.
\end{equation}
\end{Proposition}
\begin{proof}
It follows from~\eqref{bound.trivial} that $k \to +\infty$ 
as $\alpha \to -\infty$.  
Since the right-hand side of~\eqref{trans} tends to~$1$ as $k \to +\infty$,
we immediately obtain that $k = -\alpha + o(\alpha)$ as $\alpha \to -\infty$.
Let us put $c := k+\alpha$.
Subtracting~$1$ from both sides of~\eqref{trans},
we now rewrite~\eqref{trans} as follows
$$
  - \frac{c}{\alpha} 
  = \frac{2 e^{-2 R k}}{1 - e^{-2 R k}}
  \,.
$$
Since $c/\alpha = o(1)$ as $\alpha \to -\infty$,
the identity asymptotically behaves as 
$$
  - \frac{c}{\alpha}  
  = 2 \, e^{2R\alpha[1+o(1)]} + O\big(e^{4R\alpha[1+o(1)]}\big)
  \,, \qquad 
  \alpha \to -\infty
  \,.
$$
Multiplying this asymptotic identity by $e^{-2R\alpha[1+o(1)]}$
and taking the limit $\alpha \to -\infty$, we conclude with
\begin{equation*}
  c = - 2 \alpha e^{2R\alpha} + o(\alpha e^{2R\alpha})  
  \,, \qquad 
  \alpha \to -\infty
  \,,
\end{equation*}
which completes the proof of the proposition.
\end{proof}

The proposition represents an improvement upon~\eqref{as.ball} when $d=1$.
Let us emphasise that the asymptotics does not enable one to extend
the disproval of Bareket's conjecture from~\cite{FK7}
to the one-dimensional situation.
When $d=1$, the analogue of the annulus $A_{R_1,R_2}$ 
of positive radii $R_1 < R_2$
is the disconnected set $(-R_2,-R_1) \cup (R_1,R_2)$.
Since, 
$
  \lambda_1^\alpha(A_{R_1,R_2}) 
  = \lambda_1^\alpha((R_1,R_2))
  = \lambda_1^\alpha(-R',R')
$ 
with $R' := (R_2-R_1)/2$ and the one-dimensional ball~$B_R$
of radius~$R$ such that $|B_R| = |A_{R_1,R_2}|$ 
satisfy $R = R_2-R_1 > R'$, it actually follows 
from Proposition~\ref{Prop.as.1D} that 
\begin{equation}\label{Bareket.1D} 
  \lambda_1^\alpha(A_{R_1,R_2}) < \lambda_1^\alpha(B_R)
\end{equation}
for all sufficiently large negative~$\alpha$.
As a matter of fact, assuming the volume constraint $|B_R| = |A_{R_1,R_2}|$,
inequality~\eqref{Bareket.1D} does hold for \emph{all} negative~$\alpha$;
this follows from the following monotonicity result
(\cf~Theorem~\ref{Thm.monotonicity}).
\begin{Proposition}\label{Thm.monotonicity.1D}
If $\alpha < 0$, then
$$
  R \mapsto \lambda_1^\alpha((-R,R))
  \quad \mbox{is strictly increasing.}
$$
\end{Proposition}
\begin{proof}
The proof of Theorem~\ref{Thm.monotonicity} does not seem 
to entirely extend to $d=1$.
Anyway, in the present one-dimensional situation,
the monotonicity result can be deduced directly from~\eqref{trans}.
Let us set
$$
  f(k,R) := k + \alpha \, \coth(R k)
  \,.
$$
Computing the partial derivatives,
\begin{align*}
  \frac{\partial f}{\partial k}(k,R)
  &= \frac{\sinh^2(Rk)-R\alpha}{\sinh^2(Rk)}
  > \frac{R^2 k^2-R\alpha}{\sinh^2(Rk)}
  > 0
  \,,
  \\
  \frac{\partial f}{\partial R}(k,R)
  &= \frac{-k \alpha}{\sinh^2(Rk)}
  > 0 \,,
\end{align*}
we conclude from the implicit-function theorem that 
the derivative of~$k=k(R)$ with respect to~$R$ is negative.
Consequently,
$$ 
  \frac{\partial \lambda_1^\alpha((-R,R)) } {\partial R} 
  = - k(R) \, k'(R)
  > 0
  \,,
$$
which completes the proof of the proposition.
\end{proof}

We complement Proposition~\ref{Thm.monotonicity.1D}
by the asymptotic behaviour of the first Robin eigenvalue
in shrinking and expanding intervals.
\begin{Proposition}\label{Prop.shrink}
If $\alpha < 0$, then
\begin{align*}
  \lambda_1^\alpha((-R,R)) &= \frac{\alpha}{R} + o(R^{-1})
  \,, && R \to 0 \,,  
  \\
  \lambda_1^\alpha((-R,R)) &= -\alpha^2 + o(1)
  \,, && R \to \infty \,.
\end{align*}
\end{Proposition}
\begin{proof}
It follows from~\eqref{bound.trivial} that $k \to +\infty$ as $R \to 0$.  
From~\eqref{trans} we then immediately obtain 
that necessarily $k R \to 0$ as $R \to 0$.
Using the Taylor expansion of $\tanh(x)$ at $x=0$,
we then conclude from~\eqref{trans} that the asymptotic behaviour 
$$
  k^2 R \, [1 + O(kR)] = -\alpha^2
  \,, \qquad
  R \to 0 
$$ 
holds.
Taking the limit $R \to 0$ we obtain the first asymptotic 
expansion of the proposition.

As for the second asymptotic expansion, 
we first note that~$k$ converges as $R \to \infty$ 
due to Proposition~\ref{Thm.monotonicity.1D}.
Since the range of the function $x \mapsto \coth(x)$
does contain zero, the possibility of the limit being zero is excluded
by~\eqref{trans}.
Consequently, $kR \to \infty$ as $R \to \infty$
and~\eqref{trans} then yields $-k/\alpha \to 1$ 
in the limit. 
\end{proof}
\begin{Remark}\label{Rem.nrs}
It is remarkable that $\lambda_1^\alpha((-R,R))$
does not converge to zero as $R \to +\infty$, 
which is the lowest point 
in the spectrum of the ``free'' Laplacian $-\Delta_\alpha^\Real$.
Consequently, $-\Delta_\alpha^{(-R,R)}$ 
cannot converge to $-\Delta_\alpha^\Real$
in a norm-resolvent sense as $R \to +\infty$. This is still
true in higher dimensions, as may be seen from the upper bound in Theorem~\ref{Thm.ballbounds}.
\end{Remark}

For later purposes, we apply Proposition~\ref{Prop.shrink}
to the behaviour of the first Robin eigenvalue 
in long thin rectangles.
\begin{Corollary}\label{Prop.rectangle}
Let $\mathcal{R}_{a,b} := (-a,a)\times(-b,b)$ be a rectangle
of half-sides $a>0$ and $b>0$.
If $\alpha < 0$, then
\begin{align*}
  \lambda_1^\alpha(\mathcal{R}_{a,b}) &= \frac{\alpha}{b} + o(b^{-1})
  \,, && a \to \infty \,, \ b \to 0 \,,  
\end{align*}
\end{Corollary}
\begin{proof}
By separation of variables, we have
$$
  \lambda_1^\alpha(\mathcal{R}_{a,b})
  = \lambda_1^\alpha((-a,a)) + \lambda_1^\alpha((-b,b))
  \,.
$$
The asymptotic behaviour is then 
a direct consequence of Proposition~\ref{Prop.shrink}.
\end{proof}
%

\section{Fixed volume: upper bounds}
%
In this section we give a proof of Theorem~\ref{Thm.bound}.

\subsection{Star-shaped geometries}
Following~\cite{FK5}, we assume that $\Omega$ is \emph{star-shaped}
with respect to a point $\xi\in\Omega$,
\ie, for each point $x\in\partial\Omega$ the segment joining~$\xi$
with~$x$ lies in $\Omega\cup\{x\}$
and is transversal to~$\partial\Omega$ at the point~$x$.
By Rademacher's theorem, the outward unit normal vector field
$
  \nu: \partial\Omega \to \Real^d
$
can be uniquely defined almost everywhere on~$\partial\Omega$.
At those points $x\in\partial\Omega$ for which $\nu(x)$
is uniquely defined, we introduce the \emph{support function}
\begin{equation}
  h_\xi(x) := (x-\xi) \cdot \nu(x) \,,
\end{equation}
where the dot denotes the standard scalar product in~$\Real^d$.

We say that~$\Omega$ is \emph{strictly star-shaped}
with respect to the point $\xi\in\Omega$
if~$\Omega$ is star-shaped with respect to~$\xi$
and the support function is uniformly positive,
\ie,
\begin{equation}\label{Ass.strict}
  \essinf_{x\in\partial\Omega} h_\xi(x) > 0 \,.
\end{equation}
In this case, we shall denote by~$\omega$ the set of points
with respect to which~$\Omega$ is strictly star-shaped,
and define the following intrinsic quantity of the domain
\begin{equation}\label{Freitas}
  F(\Omega) := \inf_{\xi\in\omega} \int_{\partial\Omega}
  \frac{\der x}{h_\xi(x)}
  \,,
\end{equation}
where~$\der x$ denotes the surface measure of~$\partial\Omega$.
We refer to~\cite{FK5} for evaluation of the quantity~$F(\Omega)$
for certain special domains.

\subsection{Shrinking coordinates}
Given a point $\xi\in\omega$ with respect to which~$\Omega$
is strictly star-shaped, we parameterise $\Omega\setminus\{\xi\}$
by means of the mapping
\begin{equation}\label{coordinates}
  \mathcal{L}: \partial\Omega\times(0,1) \to \Real^d:
  \big\{ (x,t) \mapsto \xi + (x-\xi) t \big\}
  \,.
\end{equation}
It has been shown in~\cite{FK5} that~$\mathcal{L}$ indeed induces
a diffeomorphism, so that $\Omega\setminus\{\xi\}$ can be identified
with the Riemannian manifold $M := \big(\partial\Omega\times(0,1),G\big)$
with the induced metric $G := \nabla\mathcal{L} \cdot (\nabla\mathcal{L})^T$,
for which, in particular,
\begin{equation}\label{FK5}
  |G(x,t)| = |g(x)| \, h_\xi(x)^2 \, t^{2(d-1)}
  \qquad \mbox{and} \qquad
  G^{dd}(x,t) = h_\xi(x)^{-2}
  \,.
\end{equation}
Here $|G| := \det(G)$, $g$~denotes the metric tensor of~$\partial\Omega$
(as a submanifold of~$\Real^d$)
and $G^{ij}$ are the coefficients of the inverse metric~$G^{-1}$.
Consequently, the volume element~$\der m$
of the manifold~$M$ is decoupled as follows:
\begin{equation}\label{volume}
  \der m = h_\xi(x) \, \der x \, \wedge \, t^{d-1} \, \der t
  \,.
\end{equation}

In contrast with~\eqref{Freitas},
the integral of the support function is actually independent of~$\xi$.
Indeed, recalling~\eqref{volume},
we have the following identity for the Lebesgue measure of~$\Omega$
\begin{equation}\label{independent}
  |\Omega| =  \int_{\partial\Omega \times (0,1)} \der m
  = \frac{1}{d} \int_{\partial\Omega} h_\xi(x)\,\der x
  \,.
\end{equation}
For the $(d-1)$-dimensional Hausdorff measure of~$\partial\Omega$,
we have
$|\partial\Omega| = \int_{\partial\Omega} \der x$.

The following proposition provides a lower bound to
the geometric quantity~\eqref{Freitas}.
\begin{Proposition}\label{Prop.F.bound}
Let~$\Omega$ be a bounded strictly star-shaped domain in~$\Real^d$
with Lipschitz boundary~$\partial\Omega$. Then
\begin{equation}\label{F.bound}
  F(\Omega) \geq \frac{|\partial\Omega|^2}{d \, |\Omega|}
  \geq \frac{|\partial B|^2}{d\,|B|}
  \,,
\end{equation}
where~$B$ is a ball of the same volume as~$\Omega$.
Here the second equality is attained 
for $d$ is larger than two if, and only if, $\Omega=B$, and it holds for any
disk when $d$ is two; the first inequality is attained for any ball.
\end{Proposition}
\begin{proof}
By the Schwarz inequality,
\begin{equation}
  |\partial\Omega|^2
  = \left(
  \int_{\partial\Omega} \sqrt{h_\xi(x)} \, \frac{1}{\sqrt{h_\xi(x)}} \, \der x
  \right)^2
  \leq \int_{\partial\Omega} h_\xi(x) \, \der x
  \int_{\partial\Omega} \frac{\der x}{h_\xi(x)}
\end{equation}
for any $\xi \in \omega$.
Recalling~\eqref{independent} and minimising over $\xi \in \omega$,
we get the first inequality of~\eqref{F.bound}.
The second estimate in~\eqref{F.bound} follows by the isoperimetric inequality,
which is known to be optimal if, and only if, $\Omega=B$.
Finally, it is easy to see that
\begin{equation}\label{F.ball}
  F(B_r) = \frac{|\partial B_r|}{r}
  = \frac{|\partial B_r|^2}{d\,|B_r|}
  \,.
\end{equation}
for any ball~$B_r$ of radius~$r$.
\end{proof}

Using the identification of~$\Omega$ with~$M$,
we introduce the unitary transform
from the Hilbert space $\sii(\Omega)$ to
$\sii\big(\partial\Omega\times(0,1),\der m\big)$
by $u \mapsto u \circ \mathcal{L} =: \Psi$.
Then
\begin{equation}\label{unitary}
\begin{aligned}
  \|u\|_{\sii(\Omega)}^2
  &=
  \int_{\partial\Omega\times(0,1)} |\Psi(x,t)|^2 \, \der m
  \,, \qquad
  \\
  \|\nabla u\|_{\sii(\Omega)}^2
  &=
  \int_{\partial\Omega\times(0,1)}
  \big(\overline{\partial_i\Psi} \, G^{ij} \, \partial_j\Psi\big)(x,t) \,\der m
  \,, \qquad
  \\
  \|u\|_{\sii(\partial\Omega)}^2
  &=
  \int_{\partial\Omega} |\Psi(x,1)|^2 \, \der x
  \,,
\end{aligned}
\end{equation}
where the Einstein summation convention is assumed in the central line,
with $\partial_i\Psi := \partial\Psi/\partial \theta^i$ for $i \in \{1,\dots,d-1\}$
using some local parameterisation
$x=x(\theta^1,\dots,\theta^{d-1})$ of~$\partial\Omega$
and $\partial_d\Psi := \partial\Psi/\partial t$.

\subsection{Proof of Theorem~\ref{Thm.bound}}
With help of the unitary transform above,
we choose in~\eqref{Rayleigh} a test function $u := \Psi \circ \mathcal{L}^{-1}$,
where~$\Psi$ depends on the radial variable~$t$ only,
\ie\ $\Psi(s,t) = \psi(t)$ with any smooth function $\psi:[0,1]\to\Real$.
Using~\eqref{unitary}, \eqref{volume} and~\eqref{FK5},
we thus obtain
\begin{equation}\label{bound1}
  \lambda_1^\alpha(\Omega)
  \leq \frac{\displaystyle
  \int_{0}^{1} |\psi'(t)|^2 \, t^{d-1} \, \der t
  \int_{\partial\Omega} \frac{\der x}{h_\xi(x)}
  +\alpha \, |\psi(1)|^2 \, \int_{\partial\Omega} \der x}
  {\displaystyle
  \int_{0}^{1} |\psi(t)|^2 \, t^{d-1} \, \der t
  \int_{\partial\Omega} h_\xi(x)\,\der x}
\end{equation}
for any smooth~$\psi$ and every $\xi \in \omega$.
Minimising over $\xi \in \omega$ and recalling~\eqref{Freitas}
and~\eqref{independent}, we arrive at the bound
\begin{equation}\label{bound2}
  \lambda_1^\alpha(\Omega)
  \leq \frac{\displaystyle
  F(\Omega) \int_{0}^{1} |\psi'(t)|^2 \, t^{d-1} \, \der t
  +\alpha \, |\partial\Omega| \, |\psi(1)|^2}
  {\displaystyle
  d \, |\Omega| \int_{0}^{1} |\psi(t)|^2 \, t^{d-1} \, \der t}
\end{equation}
for any smooth~$\psi$.

Now, let~$B$ an open ball of the same volume as~$\Omega$,
\ie\ $|B|=|\Omega|$.
Employing the fact that the eigenfunction of~$-\Delta_\alpha^{B}$
corresponding to the lowest eigenvalue $\lambda_1^\alpha(B)$
is radially symmetric, we have the equality
\begin{equation}\label{ball.eq}
  \lambda_1^\alpha(B)
  = \frac{\displaystyle
  F(B) \int_{0}^{1} |{\varphi_1^\alpha}'(t)|^2 \, t^{d-1} \, \der t
  +\alpha \, |\partial B| \, |\varphi_1^\alpha(1)|^2}
  {\displaystyle
  d \, |\Omega| \int_{0}^{1} |\varphi_1^\alpha(t)|^2 \, t^{d-1} \, \der t}
\end{equation}
with some smooth function~$\varphi_1^\alpha:[0,1]\to\Real$.

Combining~\eqref{bound2} with~\eqref{ball.eq},
we get Theorem~\ref{Thm.bound}.
\hfill\qed

\subsection{Comments on Theorem~\ref{Thm.bound}}
\subsubsection{Theorem~\ref{Thm.bound} does not prove Bareket's conjecture}
By Proposition~\ref{Prop.F.bound},
\begin{equation}
  \lambda_1^\alpha(\Omega) \leq \lambda_1^{\tilde{\alpha}}(B)
\end{equation}
(recall that $\lambda_1^\alpha(\Omega)$ is non-positive for any $\alpha \leq 0$,
\cf~\eqref{bound.trivial}).
However, since Proposition~\ref{Prop.F.bound} also implies
\begin{equation}
  \tilde{\alpha}
  \geq \alpha \, \frac{|\partial B|}{|\partial\Omega|}
  \geq \alpha
  \,,
\end{equation}
where the second inequality follows from the isoperimetric inequality,
the bound~\eqref{bound} together with these estimates
does not seem to give anything useful as regards Bareket's conjecture
(that requires $\lambda_1^\alpha(\Omega) \leq \lambda_1^{\alpha}(B)$).

\subsubsection{Bound~\eqref{bound} is not optimal
in the limit $\alpha \to -\infty$}
Indeed, if~$\partial\Omega$ is smooth,
then \cite[Thm.~2.11]{Levitin-Parnovski_2008} yields
\begin{equation}
  \lambda_1^\alpha(\Omega) = -\alpha^2 + o(\alpha^2)
  \qquad \mbox{as} \qquad
  \alpha \to -\infty
  \,.
\end{equation}
Consequently, \eqref{bound}~implies
\begin{equation}
  -1 = \lim_{\alpha \to -\infty} \frac{\lambda_1^\alpha(\Omega)}{\alpha^2}
  \leq \frac{F(\Omega)}{F(B)} \,
  \lim_{\alpha \to -\infty} \frac{\lambda_1^{\tilde{\alpha}}(B)}{\alpha^2}
  = - \frac{|\partial\Omega|^2}{|\partial B|^2} \, \frac{F(B)}{F(\Omega)}
  \,.
\end{equation}
However, the right-hand side is \emph{greater} than or equal to~$-1$
due to Proposition~\ref{Prop.F.bound}.

\subsubsection{Bound~\eqref{bound} is optimal in the limit $\alpha \to 0$}
Indeed, by analytic perturbation theory
(see, \eg, \cite{Lacey-Ockendon-Sabina_1998}),
\begin{equation}\label{LOS}
  \lambda_1^\alpha(\Omega) = \frac{|\partial\Omega|}{|\Omega|} \, \alpha
  + O(\alpha^2)
  \qquad \mbox{as} \qquad
  \alpha \to 0
  \,.
\end{equation}
Consequently, \eqref{bound}~implies
\begin{equation}
  \frac{|\partial\Omega|}{|\Omega|}
  = \lim_{\alpha \to 0-} \frac{\lambda_1^\alpha(\Omega)}{\alpha}
  \geq \frac{F(\Omega)}{F(B)} \,
  \lim_{\alpha \to 0-} \frac{\lambda_1^{\tilde{\alpha}}(B)}{\alpha}
  = \frac{|\partial\Omega|}{|B|}
\end{equation}
and it remains to recall that $|\Omega|=|B|$.
Moreover, since
\begin{equation}
  \frac{|\partial\Omega|}{|B|}
  \geq \frac{|\partial B|}{|B|}
  = \lim_{\alpha \to 0-} \frac{\lambda_1^\alpha(B)}{\alpha}
  \,,
\end{equation}
the argument also implies that~\eqref{bound} yields
the validity of $\lambda_1^\alpha(\Omega) \leq \lambda_1^{\alpha}(B)$
for all the negative~$\alpha$ which have sufficiently small~$|\alpha|$
(with the smallness depending on~$\Omega$).

\subsubsection{Bound~\eqref{bound} is optimal for long thin rectangles}
\label{Sec.comp}
Let $\mathcal{R}_{a,b} := (-a,a)\times(-b,b)$ be a rectangle
of half-sides $a>0$ and $b>0$ and let us compare the upper bound
of Theorem~\ref{Thm.bound} with the actual eigenvalue of the rectangle.
We are interested in the regime when $a \to \infty$ and $b \to 0$,
while keeping the area of the rectangle fixed, 
say $|\mathcal{R}_{a,b}| = 4 a b = 1$.
All the asymptotic formulae below in this subsection
are with respect to this limit.

The geometric quantity $F(\Omega)$ was computed in~\cite{FK5}
for various domains~$\Omega$ including rectangular parallelepipeds 
and ellipsoids. In particular, \cite[Ex.~1]{FK5} 
and \cite[Ex.~2]{FK5} (see also~\eqref{F.ball})
respectively yield 
$$
  F(\mathcal{R}_{a,b}) = |\mathcal{R}_{a,b}| \, (a^{-2}+b^{-2}) 
  \qquad \mbox{and} \qquad
  F(B) = 2 \pi
  \,,
$$
where~$B$ is the disk of the same area as $\mathcal{R}_{a,b}$,
\ie~the radius of~$B$ equals $1/\sqrt{\pi}$.
(Note that $F(B_R)$ is independent of~$R$ if $d=2$.)

We have $|\partial \mathcal{R}_{a,b}| = 4 (a+b) = O(a)$ 
and $F(\mathcal{R}_{a,b}) = O(a^2)$.
Consequently, the ratio
$|\partial \mathcal{R}_{a,b}|/F(\mathcal{R}_{a,b})$ 
behaves as $O(a^{-1})$
and the same decay rate holds for~$\tilde{\alpha}$ from~\eqref{bound}.
Using~\eqref{LOS}, we therefore have
$$
  \lambda_1^{\tilde{\alpha}}(B) 
  = \frac{|\partial B|}{|B|} \, \tilde{\alpha} + O(\tilde{\alpha}^2)
  = \frac{|\partial \mathcal{R}_{a,b}|}{|B|} 
  \frac{F(B)}{F(\mathcal{R}_{a,b})}
  \, \alpha + O(a^{-2})
  \,.
$$
Consequently, the upper bound for $\lambda_1^{\alpha}(\mathcal{R}_{a,b})$ 
given by Theorem~\ref{Thm.bound} reads
\begin{equation}\label{comp1}
  \frac{F(\mathcal{R}_{a,b})}{F(B)} \, \lambda_1^{\tilde{\alpha}}(B)
  = \frac{|\partial \mathcal{R}_{a,b}|}{|B|} \, \alpha 
  + \frac{F(\mathcal{R}_{a,b})}{F(B)} \, O(a^{-2})
  = 4 a \alpha + O(1)
  \,.
\end{equation}
At the same time, Corollary~\ref{Prop.rectangle} implies
the exact asymptotics
\begin{equation}\label{comp2}
  \lambda_1^{\alpha}(\mathcal{R}_{a,b}) 
  = 4 a \alpha + o(a)
  \,.
\end{equation}
We see that the leading orders of~\eqref{comp1} and~\eqref{comp2} coincide.

\section{Fixed perimeter: the disk maximises the first eigenvalue}
%
In this section we restrict to $d=2$.
Moreover, we assume that the bounded domain
$\Omega \subset \Real^2$ is of class~$C^2$.
Then the boundary~$\partial\Omega$ will, in general,
be composed of a finite union of $C^2$-smooth Jordan curves
$\Gamma_0, \Gamma_1, \dots, \Gamma_N$, $N \geq 0$,
where~$\Gamma_0$ is the \emph{outer boundary}
\ie~$\Omega$ lies in the \emph{interior}~$\Omega_0$ of~$\Gamma_0$.
If $N=0$, then~$\Omega$ is simply connected and $\Omega=\Omega_0$.

We denote by $L_\mathrm{tot}:=|\partial\Omega|$
and $L_0:=|\Gamma_0|$
the perimeter and \emph{outer perimeter} of~$\Omega$, respectively,
where~$|\cdot|$ stands for the $1$-dimensional Hausdorff measure.
By the isoperimetric inequality, we have
%
$
  L_0^2 \geq 4 \pi A_\mathrm{tot}
$,
%
where $A_\mathrm{tot}:=|\Omega|$ denotes the area of~$\Omega$.
Here~$|\cdot|$ stands for the $2$-dimensional Lebesgue measure.

Under our regularity assumptions,
the operator domain
$
  \Dom(-\Delta_\alpha^\Omega)
$
consists of functions $u \in W^{2,2}(\Omega)$
which satisfy the Robin boundary conditions
of~\eqref{eq:robin} in the sense of traces
and the boundary value problem~\eqref{eq:robin}
can be thus considered in a classical setting.

\subsection{An upper bound from~\texorpdfstring{\cite{FK7}}{FK7}}
As in~\cite{FK7}, the main ingredient in our proof of Theorem~\ref{Thm.perimeter}
is the method of parallel coordinates that was originally used
by Payne and Weinberger~\cite{Payne-Weinberger_1961}
in the Dirichlet case
(which formally corresponds to $\alpha = +\infty$
in the present setting).
It consists in choosing a test function in~\eqref{Rayleigh}
whose level lines are parallel to the boundary of~$\Omega$.
As a matter of fact, since~$\alpha$ is non-positive,
it is possible to base the coordinates on
the outer component~$\Gamma_0 \subset \partial\Omega$ only.
The result of this approach is the following theorem
established in~\cite{FK7}.

\begin{Theorem}[\cite{FK7}]\label{Thm.better}
Let $\alpha \leq 0$.
For any bounded planar domain~$\Omega$ of class~$C^2$,
$$
  \lambda_1^\alpha(\Omega) \leq \mu_1^\alpha(A_{R_1,R_2})
  \,,
$$
where~$\mu_1^\alpha(A_{R_1,R_2})$ is the first eigenvalue
of the Laplacian in the annulus $A_{R_1,R_2}$ with radii
\begin{equation}\label{radii}
  R_1 := \frac{\sqrt{L_0^2-4\pi A_\mathrm{tot}}}{2\pi}
  \,, \qquad
  R_2 := \frac{L_0}{2\pi} 	
  \,,
\end{equation}
subject to the Robin boundary condition with~$\alpha$
on the outer circle and the Neumann boundary condition
on the inner circle.
\end{Theorem}

We note that~$A_{R_1,R_2}$ has the same area as~$\Omega$,
\ie\ $|A_{R_1,R_2}|=A_\mathrm{tot}=|\Omega|$.
Using the rotational symmetry and polar coordinates,
the variational characterisation of
the Robin-Neumann eigenvalue reads
\begin{equation}\label{Rayleigh.annulus}
  \mu_1^\alpha(A_{R_1,R_2}) =
  \inf_{\stackrel[\psi\not=0]{}{\psi \in W^{1,2}((R_1,R_2))}}
  \frac{\displaystyle
  \int_{R_1}^{R_2} \psi'(r)^2 \, r \, \der r
  +\alpha \, R_2 \, \psi(R_2)^2}
  {\displaystyle
  \int_{R_1}^{R_2} \psi(r)^2 \, r \, \der r}
  \,.
\end{equation}

\subsection{Proof of Theorem~\ref{Thm.perimeter}}
The following upper bound represents a crucial step
in our proof of Theorem~\ref{Thm.perimeter}.
\begin{Proposition}\label{Prop.test}
Let $\alpha \leq 0$. For any $0<R_1<R_2$, we have
$$
  \mu_1^\alpha(A_{R_1,R_2}) \leq \lambda_1^\alpha(B_{R_2})
  \,,
$$
where~$B_{R_2}$ denotes the disk of radius~$R_2$.
\end{Proposition}
\begin{proof}
By symmetry,
$\lambda_1^\alpha(B_{R_2})$ is the smallest solution of
the one-dimensional boundary-value problem
\begin{equation}
\left\{
\begin{aligned}
  -r^{-1}[r \phi'(r)]' &= \lambda \;\! \phi(r) \,,
  && r \in [0,R_2] \,,
  \\
  \phi'(0) &= 0 \,,
  \\
  \phi'(R_2) + \alpha \;\! \phi(R_2) &=0 \,.
\end{aligned}
\right.
\end{equation}
The associated eigenfunction~$\phi_1$ can be chosen to be positive
and normalised to one in $\sii((0,R_2),r\,\der r)$.
Using~$\phi_1$ as a test function in~\eqref{Rayleigh.annulus}
and integrating by parts, we obtain
\begin{equation}\label{test.bound}
  \mu_1^\alpha(A_{R_1,R_2})
  \leq \lambda_1^\alpha(B_{R_2}) - R_1 \, \phi(R_1) \, \phi'(R_1)
  \,.
\end{equation}
At the same time, using the differential equation in~\eqref{elliptic},
we have
\begin{equation*}
  [r \, \phi_1(r) \, \phi_1'(r)]'
  = -\lambda_1^\alpha(B_{R_2}) \, r \, \phi_1(r)^2 + r \, \phi_1'(r)^2
  \geq 0
\end{equation*}
for all $r \in [0,R_2]$,
where the inequality follows from the fact
that $\lambda_1^\alpha(B_{R_2})$ is non-positive, 
\cf~\eqref{bound.trivial}.
Hence, $r \mapsto r \, \phi_1(r) \, \phi_1'(r)$
is non-decreasing and the proposition follows
as a consequence of~\eqref{test.bound}.
\end{proof}

Clearly the area of~$B_{R_2}$ is greater than the area of~$A_{R_1,R_2}$.
On the other hand, the perimeter of~$B_{R_2}$ is less than
the perimeter of~$A_{R_1,R_2}$, indeed
$$
  |\partial B_{R_2}| = L_0
  \,,
$$
the outer perimeter of~$\Omega$.
Hence, Theorem~\ref{Thm.better} together with Proposition~\ref{Prop.test}
immediately implies Theorem~\ref{Thm.perimeter}
for simply connected domains~$\Omega$, when $L_\mathrm{tot}=L_0$.

To conclude the proof of Theorem~\ref{Thm.perimeter}
in the general case of possibly multiply connected domains,
we recall the monotonicity result of Theorem~\ref{Thm.monotonicity}.
Consequently,
\begin{equation}\label{bound.consequence}
  \lambda_1^\alpha(B_{R_2})
  \leq \lambda_1^\alpha(B_{R_3})
  \,, \qquad \mbox{where} \qquad
  r_3 := \frac{L_\mathrm{tot}}{2\pi}
  \,,
\end{equation}
for all $\alpha \leq 0$ (the statement is trivial for $\alpha=0$),
where~$R_3$ is chosen in such a way that~$B_{R_3}$
has the same perimeter as~$\Omega$.
Summing up, Theorem~\ref{Thm.perimeter} is proved
as a consequence of Theorem~\ref{Thm.better}, Proposition~\ref{Prop.test}
and~\eqref{bound.consequence}.
\hfill\qed

\subsection{Comments on Theorem~\ref{Thm.perimeter}}
Combining Theorem~\ref{Thm.perimeter} with Theorem~\ref{Thm.ballbounds},
we get an explicit upper bound
$$
  \lambda_1^\alpha(\Omega) < -\alpha^2 + \frac{2\pi}{|\partial\Omega|} \, \alpha
$$
for every $\alpha<0$ and all bounded planar domains~$\Omega$ of class~$C^2$.
>From Theorem~\ref{Thm.better} and Proposition~\ref{Prop.test},
we know that the total perimeter $|\partial\Omega|=L_\mathrm{tot}$
can be replaced by the outer perimeter~$L_0$.

\section{Numerical results}
%
\subsection{Extremal domains}
In this section we present the main results that we gathered for the optimisation of Robin eigenvalues with negative parameter.
In all the numerical simulations we considered domains with unit area and the eigenvalues were calculated using the method of
fundamental solutions~\cite{Alves-Antunes_2005,Alves-Antunes_2013}. 
The maximisation of Robin eigenvalues was solved by a gradient type method involving Hadamard shape derivatives which allows to
minimise a sequence of functionals
\[
  \mathcal{F}_m(\Omega)
  := -\lambda_n^\alpha(\Omega)+c_m\left(|\Omega|-1\right)^2
\]
for a gradually increasing sequence of penalty parameters $c_m$. 

We assume that $\lambda_n \equiv\lambda_n^\alpha(\Omega)$ is simple, 
$u$~is an associated normalised real-valued eigenfunction 
and use the notation $\Omega(t):=(I+tV)(\Omega)$, 
where~$I$ is the identity and~$V$ is a given deformation field. 
The Hadamard shape derivative for simple Robin eigenvalues
is given by (see, \eg, \cite[Ex.~3.5]{Henry})
\[
  \frac{\partial}{\partial t}\lambda_n^\alpha(\Omega(t))\Big|_{t=0}
  = \int_{\partial\Omega} \left[
  |\nabla_{\partial\Omega}u|^2
  -\left(\lambda_n^\alpha (\Omega)+\alpha^2-H\alpha\right)u^2
  \right]
  V \cdot\nu
  \,,
\]
where $H :=\divergence \nu$ is the mean curvature of~$\partial\Omega$ 
and $\nabla_{\partial\Omega}u$ is the tangential component of the gradient of~$u$.
We note that special variants of this formula 
for homothetic deformations of balls 
can be found in Section~\ref{Sec.monotonicity}.

In the optimisation of each of the eigenvalues $\lambda_n$ we considered the
case of connected and disconnected domains with up to~$n$ connected components. 
In the latter the optimisation was performed at all the
connected components. At the end, the optimal eigenvalue was the maximal eigenvalue obtained from all the cases.

Figure~\ref{fig:lambda1}-left shows the results for the maximisation of the
first Robin eigenvalue as a function of the Robin parameter $\alpha$. In blue we plot the first eigenvalue of the disk of unit area
while similar results for the maximal eigenvalue among annuli are shown in red. We will denote by $A_{n,\alpha}^\ast$ the
maximiser of the $n$-th eigenvalue in the class of the annuli of unit area, for a given Robin parameter $\alpha$. Although the two
graphs are quite close, it does turn out that while the disk is the (unique) global maximiser for $\alpha\in(\alpha_1^\ast,0)$ with
$\alpha_1^\ast\approx-7.2875$, there is a transition at $\alpha=\alpha_{1}^{\ast}$ where this role is then
taken by $A_{1,\alpha}^\ast$. For $\alpha=\alpha_1^\ast$ we have non-uniqueness of the maximiser.
In the right plot of the same figure we show the difference $\lambda_1\left(A_{1,\alpha}^\ast\right)-\lambda_1(B)$, for the
region~$\alpha\in[-7.5,-7.15]$. 
\begin{figure}[ht]
\centering
(a)\includegraphics[width=0.44\textwidth]{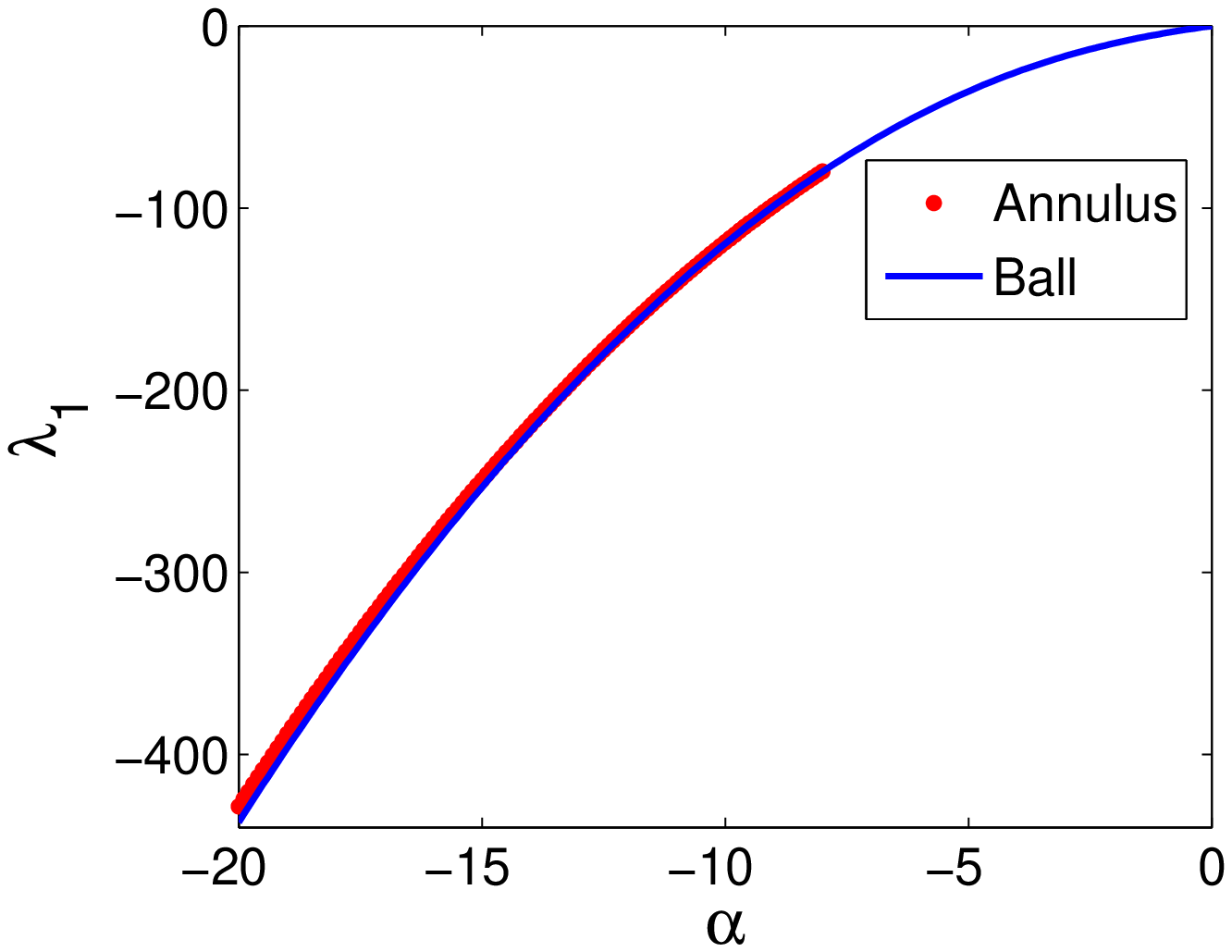}
(b)\includegraphics[width=0.44\textwidth]{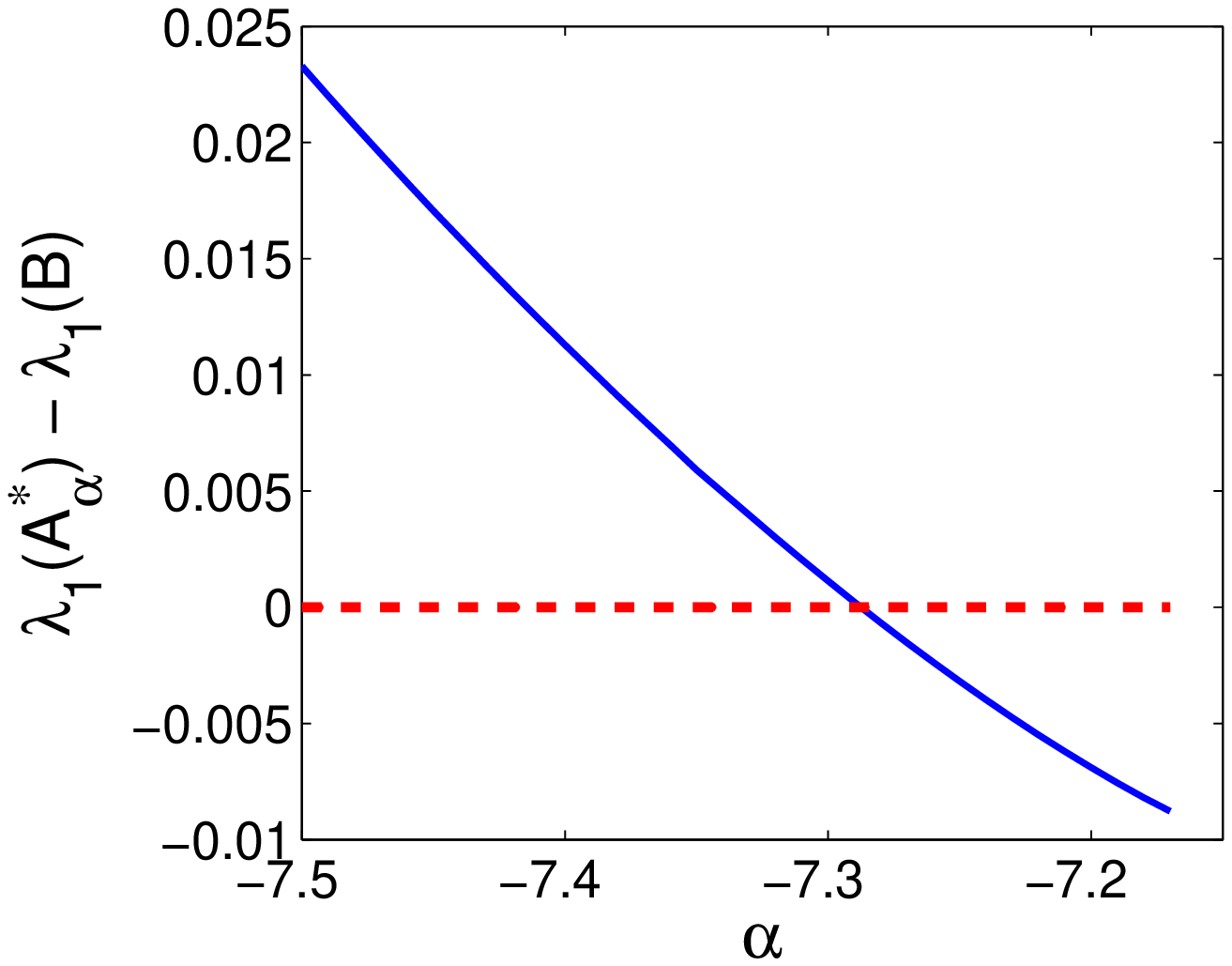}
\caption{(a) Optimal first Robin eigenvalue. The maximiser is an annulus for $\alpha<\alpha_1^\ast\approx-7.2875$ and the ball for
$\alpha<\alpha_1^\ast\leq0$. (b) Plot of the difference $\lambda_1\left(A_{1,\alpha}^\ast\right)-\lambda_1(B)$, for $\alpha\in[-7.5,-7.15]$.}
\label{fig:lambda1}
\end{figure}

In Figure~\ref{fig:lambda2} we plot the maximal second Robin eigenvalue. Again, the maximiser is an annulus $A_{2,\alpha}^\ast$ for
$\alpha<\alpha_2^\ast\approx-6.4050$ and the ball for $\alpha>\alpha_2^\ast$.

\begin{figure}[ht]
\centering
\includegraphics[width=0.46\textwidth]{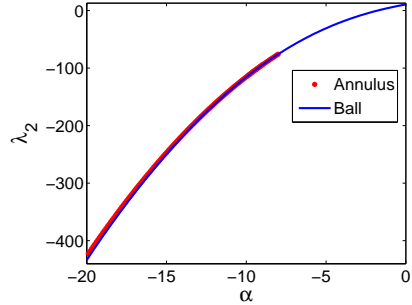}
\caption{Optimal second eigenvalue. The maximiser is an annulus $A_{2,\alpha}^\ast$ for $\alpha<\alpha_2^\ast\approx-6.4050$ and the ball for $\alpha>\alpha_2^\ast$.}
\label{fig:lambda2}
\end{figure}

In Figure~\ref{fig:lambda1rad} we plot the inner radius $R_1$ of the optimal annuli $A_{1,\alpha}^\ast$ and $A_{2,\alpha}^\ast$, as a
function of $\alpha$. We observe that in both cases the inner radius decreases with the decrease in $\alpha$.

\begin{figure}[ht]
\centering
\includegraphics[width=0.46\textwidth]{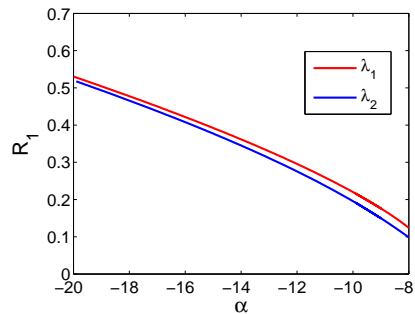}
\caption{Inner radius of the optimal annuli $A_{1,\alpha}^\ast$ and $A_{2,\alpha}^\ast$, as a function of $\alpha$.}
\label{fig:lambda1rad}
\end{figure}

In Figure~\ref{fig:lambda3} we plot the optimal third eigenvalue for some Robin parameters $\alpha$. For a comparison, we include also the third
eigenvalue of the ball and of the union of two balls of the same area. In this case, our numerical results suggest that the maximisers are always
connected, for an arbitrary $\alpha<0$. Moreover, these maximisers degenerate to two balls for $\alpha=0$, which is the conjectured maximiser
in the case of Neumann boundary conditions 
-- see~\cite{Girouard-Nadirashvili-Polterovich_2009} 
where it is shown that the third eigenvalue of
simply connected domains with Neumann boundary conditions 
never exceeds this value. 
In the right plot of the same figure we show the maximisers for
$\alpha=-14,-8,-1$.

\begin{figure}[ht]
\centering
\includegraphics[width=0.44\textwidth]{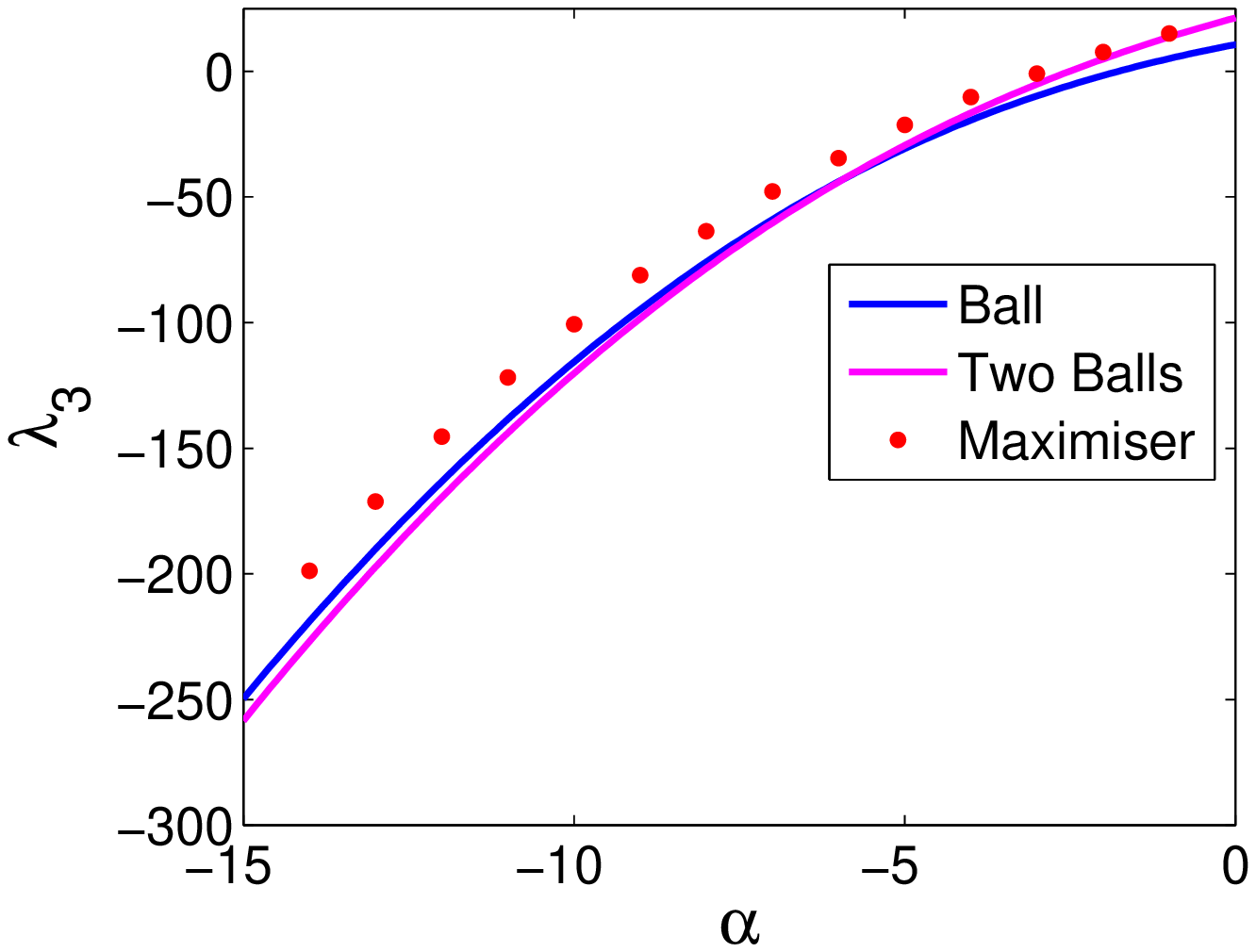}
\includegraphics[width=0.44\textwidth]{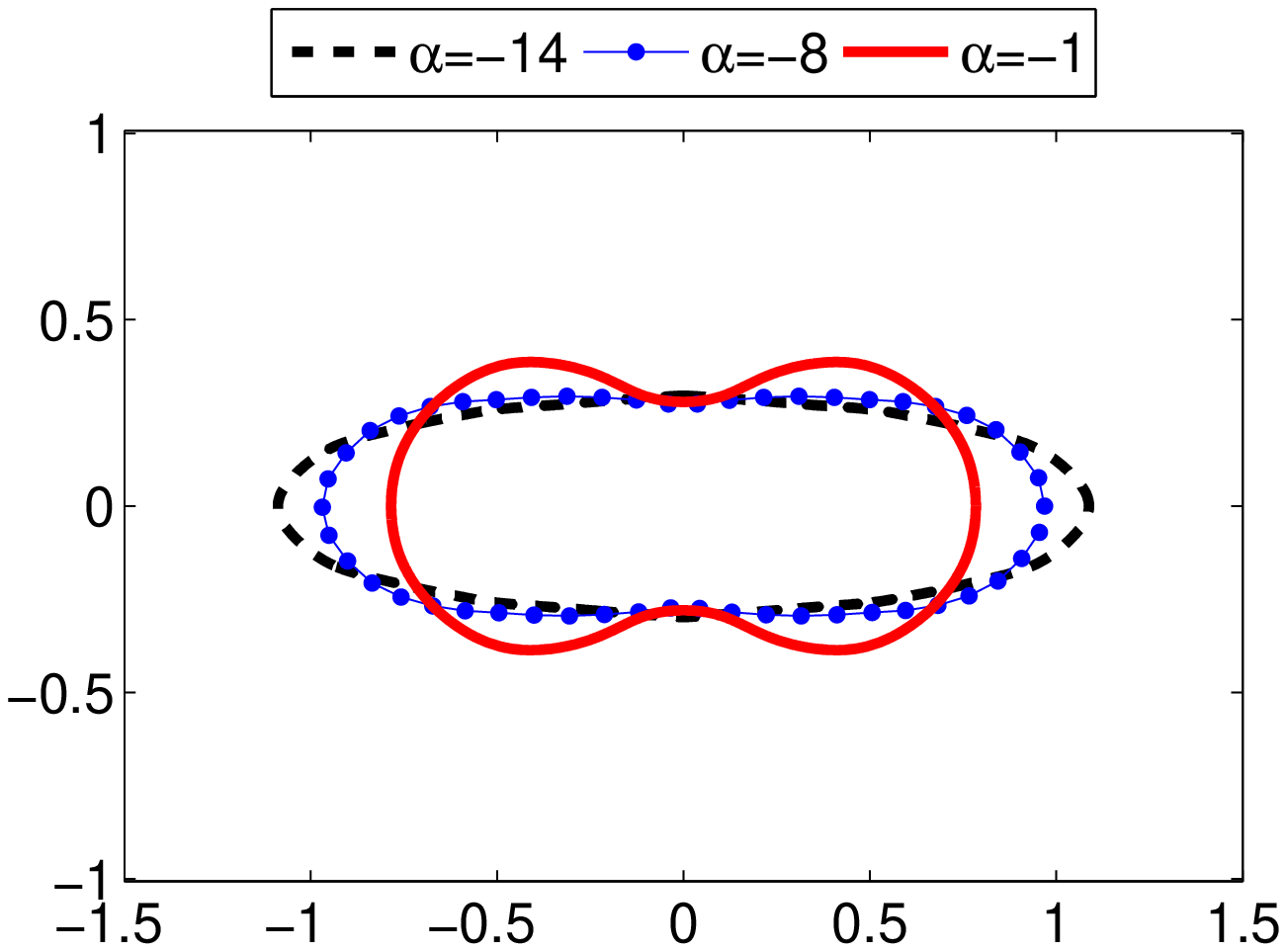}
\caption{(a) The optimal third eigenvalue, together with the third eigenvalue of the ball and of the union of two balls of the same area. (b)
The maximisers for $\alpha=-14,-8,-1$.}
\label{fig:lambda3}
\end{figure}

In Figure~\ref{fig:lambda4} we show the numerical maximisers obtained for
$\alpha=-0.25,-4$, $-8,-13$.

\begin{figure}[ht]
\centering
\includegraphics[width=0.44\textwidth]{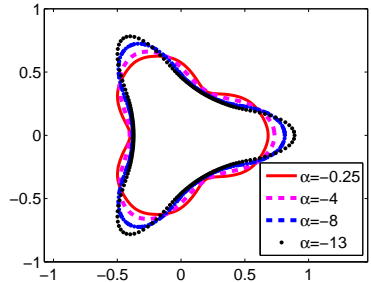}
\caption{Maximisers for the fourth eigenvalue for $\alpha=-13,-8-,4,-0.25$.}
\label{fig:lambda4}
\end{figure}

We shall now present the numerical results obtained 
for three-dimensional domains. We denote by $A_{n,\alpha}^\ast$ the maximiser of
the $n$-th eigenvalue within the class of spherical shells of unit volume, 
for a given value of the Robin parameter $\alpha$.
In Figure~\ref{fig:lambda1-3D} we plot the maximal first and second Robin eigenvalues. 
In this case, the ball is the maximiser of $\lambda_n$,
$n=1,2$, for $\alpha\in(\alpha_n^\ast,0)$, 
while for $\alpha<\alpha_n^\ast$ the maximiser becomes $A_{n,\alpha}^\ast$, 
where the values
of $\alpha_n^\ast$ obtained numerically 
are $\alpha_n^\ast\approx-1.7149$ and $\alpha_2^\ast\approx-5.6637$. 
We also considered the maximisation of the first Robin eigenvalue among domains with a given surface area. In this case our numerical results suggest that the ball is always the maximiser.

\begin{figure}[ht]
\centering
(a)\includegraphics[width=0.44\textwidth]{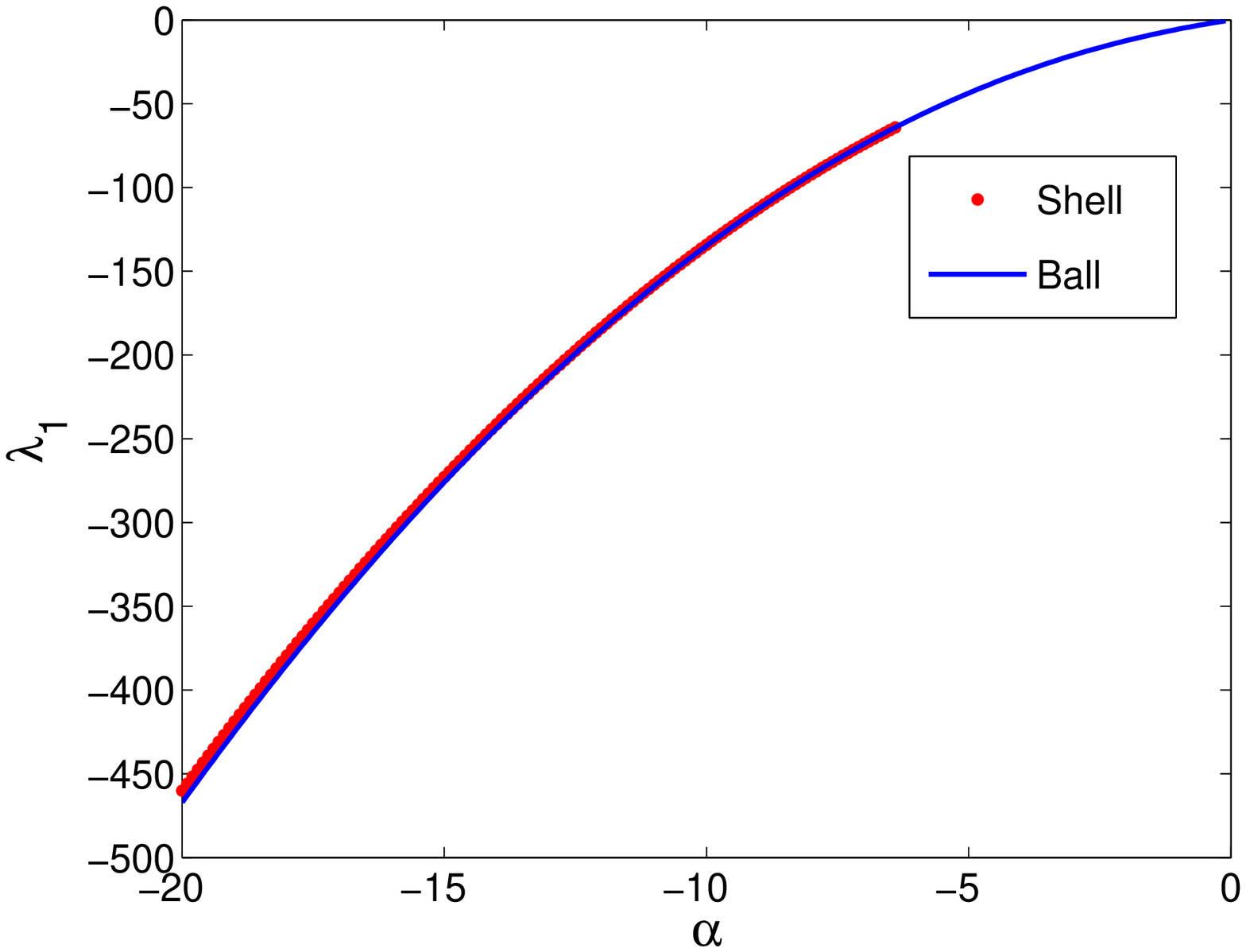}
(b)\includegraphics[width=0.44\textwidth]{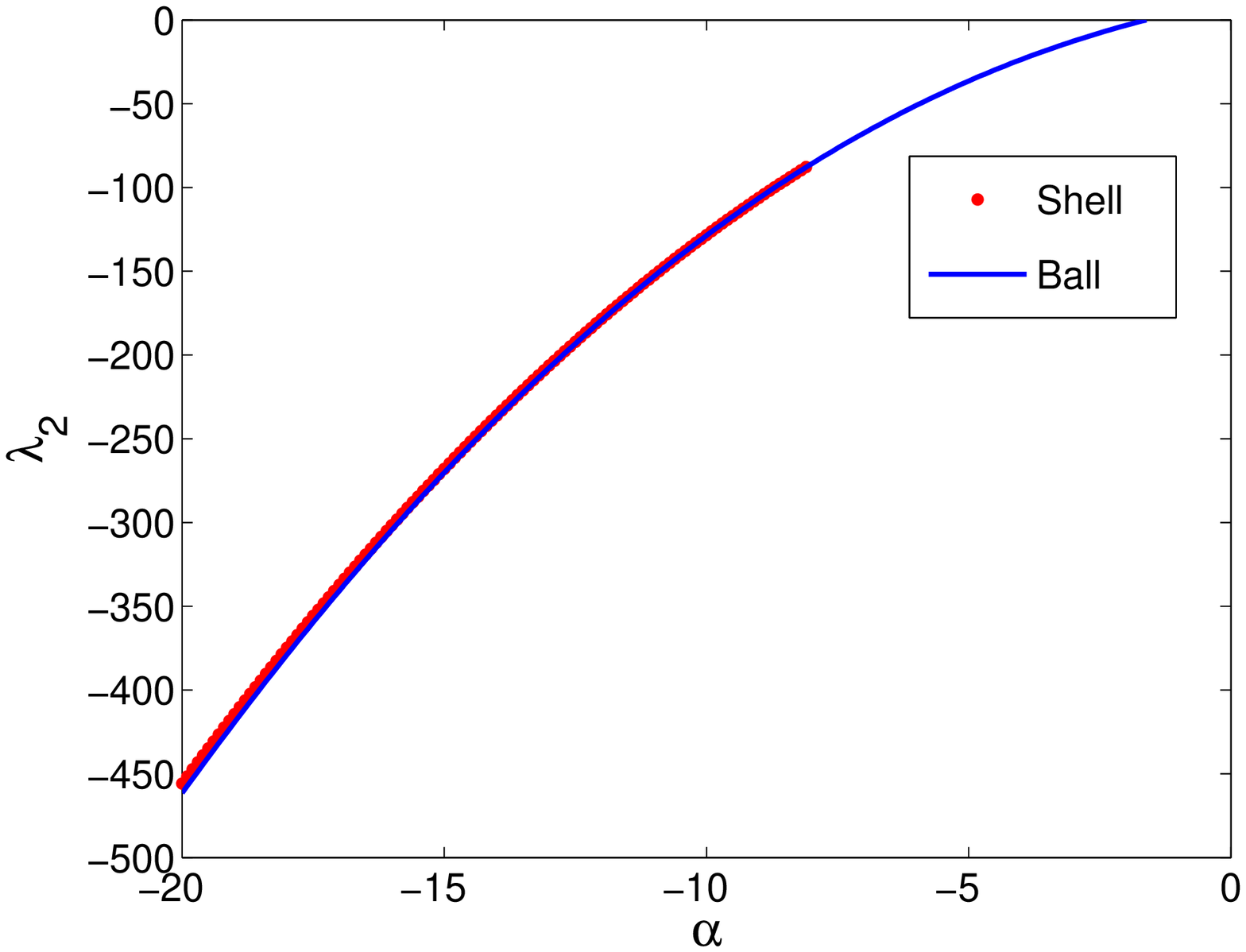}
\caption{(a) Optimal first (left plot) and second (right plot) 
Robin eigenvalues for three-dimensional domains.}
\label{fig:lambda1-3D}
\end{figure}

Finally, and in order to understand the behaviour of the bifurcation point where the switching between balls and shells takes place, we
analysed numerically the equations determining the eigenvalues of these two domains where we now take the dimension variable
to change continuously between $2$ and $6$. From Figure~\ref{fig:alphacrit} we see that both the critical value where the bifurcation occurs
and the corresponding
value of the inner radius increase as the dimension increases.

\begin{figure}[ht]
\centering
(a)\includegraphics[width=0.44\textwidth]{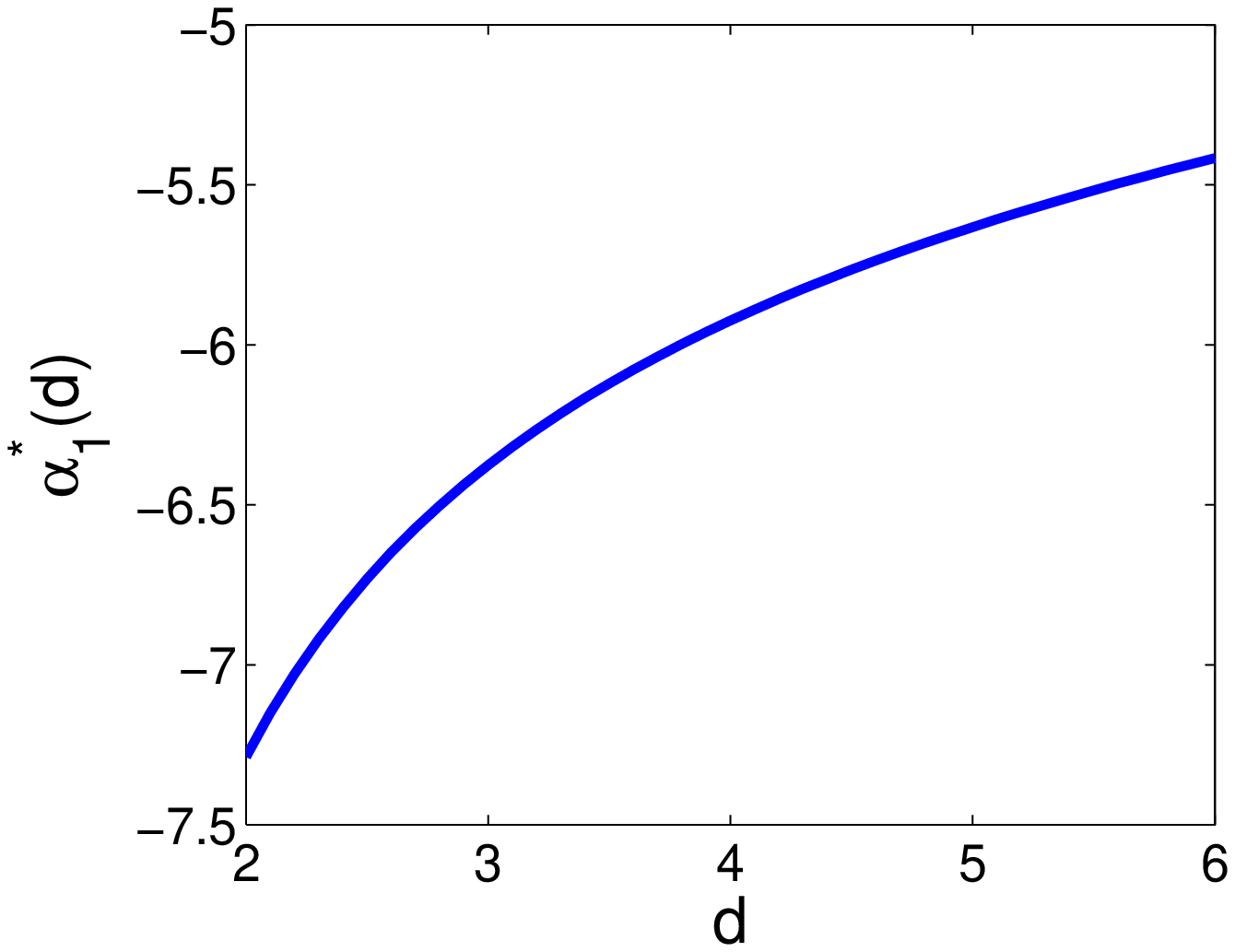}
(b)\includegraphics[width=0.44\textwidth]{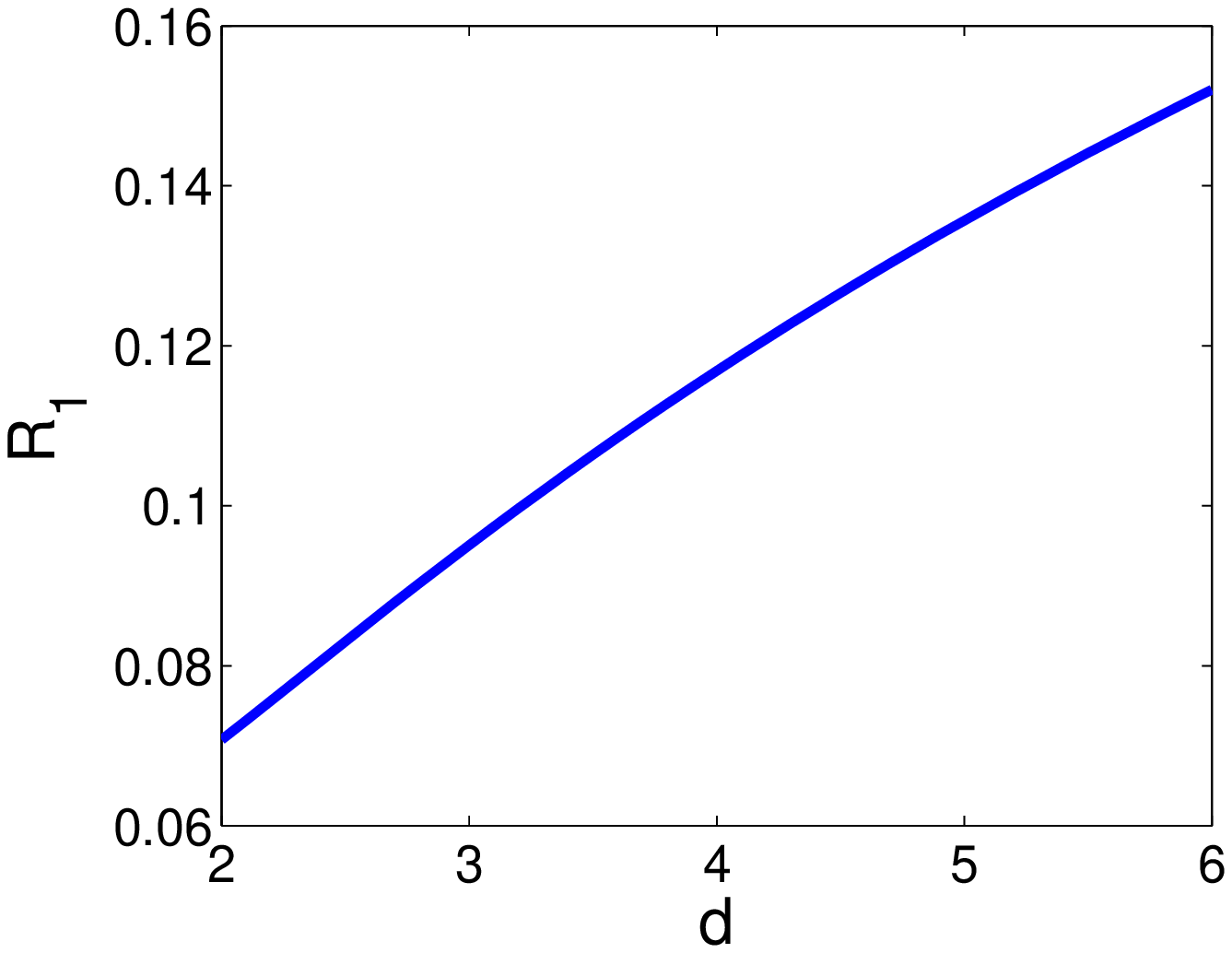}
\caption{(a) Critical value of $\alpha_1^\ast$
where the first eigenvalue of balls becomes smaller than that of the optimal shell with the same volume.
(b) Corresponding value of the smallest radius at the bifurcation.}
\label{fig:alphacrit}
\end{figure}

\subsection{Evaluation of upper bounds}
In this section we test the bounds provided by Theorems~\ref{Thm.bound},~\ref{Thm.perimeter} and~\ref{Thm.better}. For a given domain $\Omega$ we define the percentage errors associated with
these bounds respectively by
\begin{align*}
  P_1
  &:= 100 \, \frac{\left|\lambda_1^\alpha(\Omega)-\frac{F(\Omega)}{F(B)} \, \lambda_1^{\tilde{\alpha}}(B)\right|}{\left|\lambda_1^\alpha(\Omega)\right|}, \,
  \\
  P_2 
  &:= 100 \, \frac{\left|\lambda_1^\alpha(\Omega)
  -\lambda_1^\alpha\big(B_{\frac{|\partial\Omega|}{2\pi}}\big)\right|}
  {\left|\lambda_1^\alpha(\Omega)\right|}
  \,,
  \\
  P_3
  &:=100 \, \frac{\left|\lambda_1^\alpha(\Omega)-\mu_1^\alpha(A_{R_1,R_2})\right|}{\left|\lambda_1^\alpha(\Omega)\right|}  
  \,.
\end{align*}
In order to illustrate the accuracy of the bounds we shall consider ellipses, ellipsoids, rectangles and parallelepipeds.

Figure~\ref{fig:bound2del} shows the percentage errors $P_1$, $P_2$ and $P_3$ obtained in the class of ellipses as a function of the
eccentricity, for $\alpha=-10,-1$.

\begin{figure}[ht]
\centering
\includegraphics[width=0.48\textwidth]{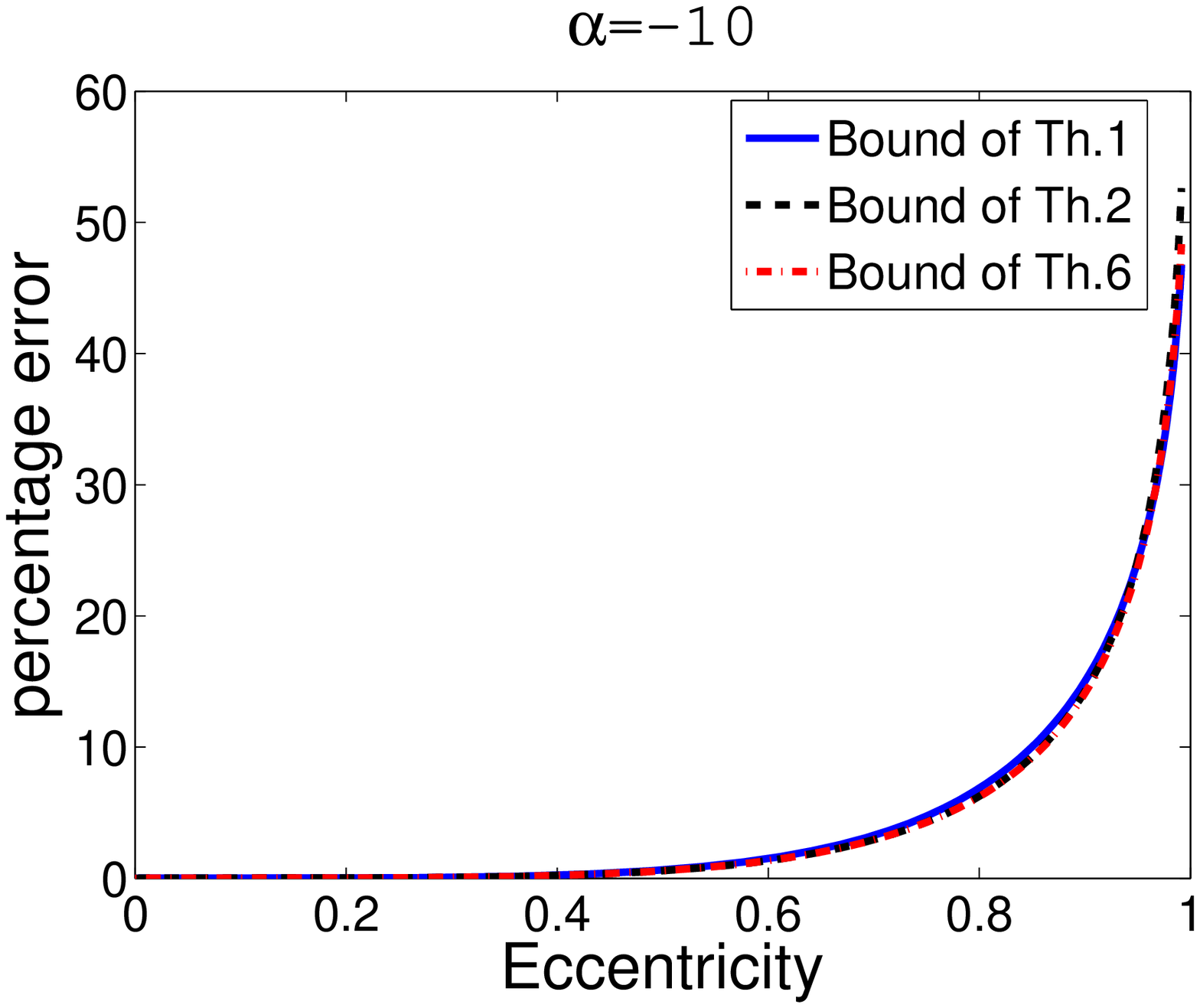}
\includegraphics[width=0.48\textwidth]{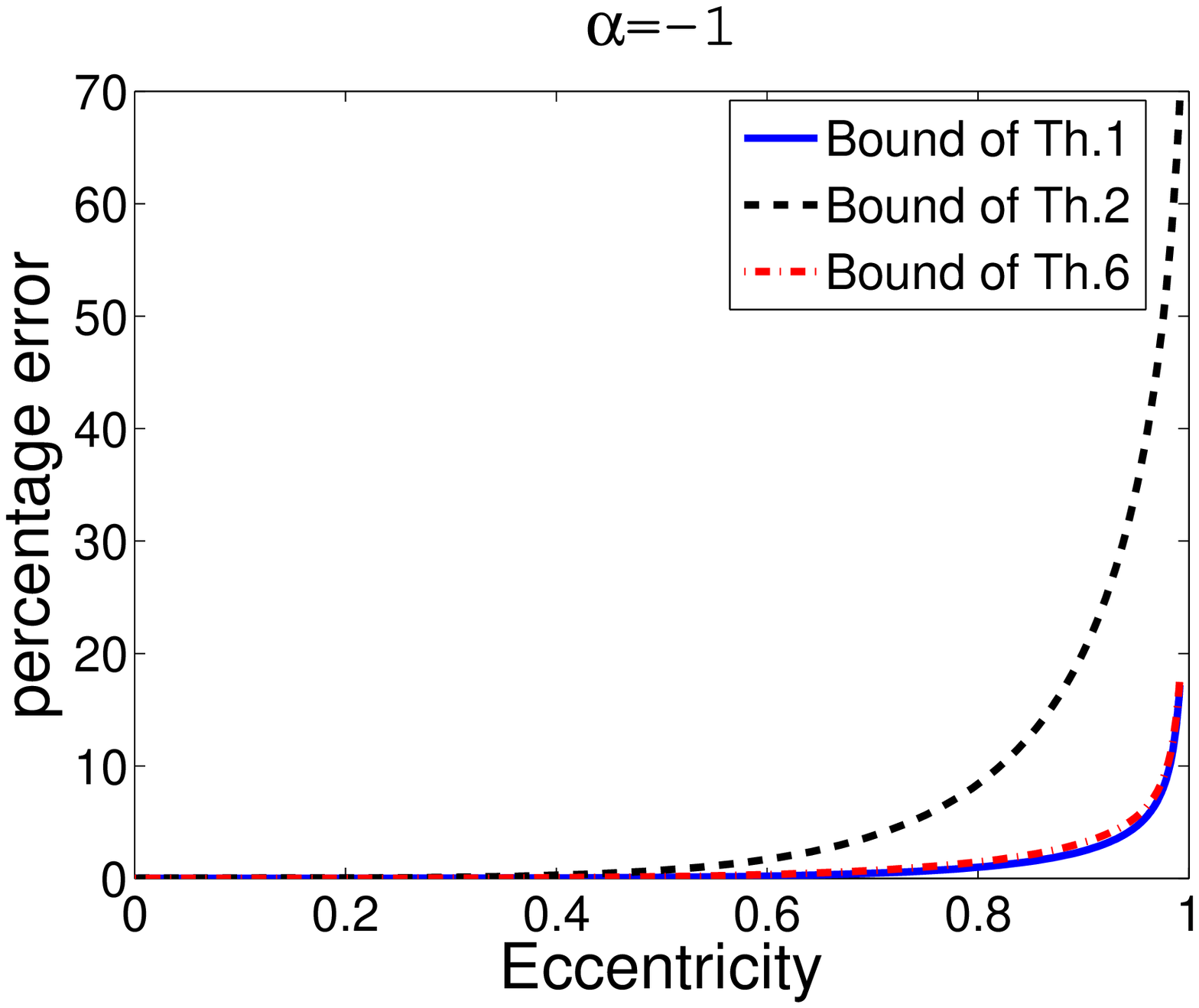}
\caption{Plots of the percentage errors associated to the bounds of Theorems~\ref{Thm.bound},~\ref{Thm.perimeter} and~\ref{Thm.better} in the class of the ellipses, as a function of the eccentricity, for $\alpha=-10,-1$.}
\label{fig:bound2del}
\end{figure}

Figure~\ref{fig:bound2d} shows the percentage errors $P_1$ obtained in the class of the rectangles  for $\alpha=-10,-5,-1$, as a function of $Q:=1-1/L^2$, where $L$ is
the length of the largest side of the rectangle. Note that the percentage errors of the bound of Theorem~\ref{Thm.bound} converge to zero, as $Q$ converge to 1. This
means that besides the balls, the bound of Theorem~\ref{Thm.bound} gives also equality asymptotically for thin rectangles.
This numerical observation is consistent with the analysis
made in Section~\ref{Sec.comp}. 
Indeed, from~\eqref{comp1} and~\eqref{comp2} there
it follows that $P_1 = o(1)$ as $L \to \infty$,
while keeping the area of the rectangle equal to one.

\begin{figure}[ht]
\centering
\includegraphics[width=0.48\textwidth]{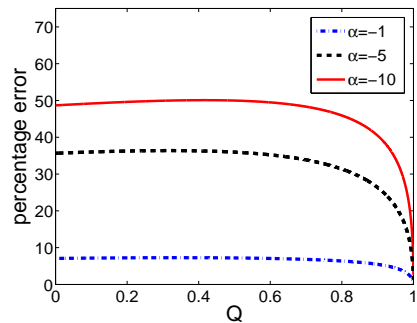}
\caption{Plot of the percentage error associated to the bound of Theorem~\ref{Thm.bound} in the class of the rectangles, as a function of $Q$.}
\label{fig:bound2d}
\end{figure}

In the three-dimensional case, 
for a given ellipsoid with semi-axes lengths $0<a\leq b\leq c$, 
we define the quantities
$\epsilon_1:=\sqrt{1-(b/c)^2}$
and $\epsilon_2:=\sqrt{1-(a/c)^2}$, 
for which we have
$0\leq\epsilon_1\leq\epsilon_2<1$. In Figure~\ref{fig:bound3d} we plot the percentage errors obtained for ellipsoids, as a function of $\epsilon_1$
and $\epsilon_2$, for $\alpha=-10,-5,-1$.

\begin{figure}[ht]
\centering
\includegraphics[width=0.32\textwidth]{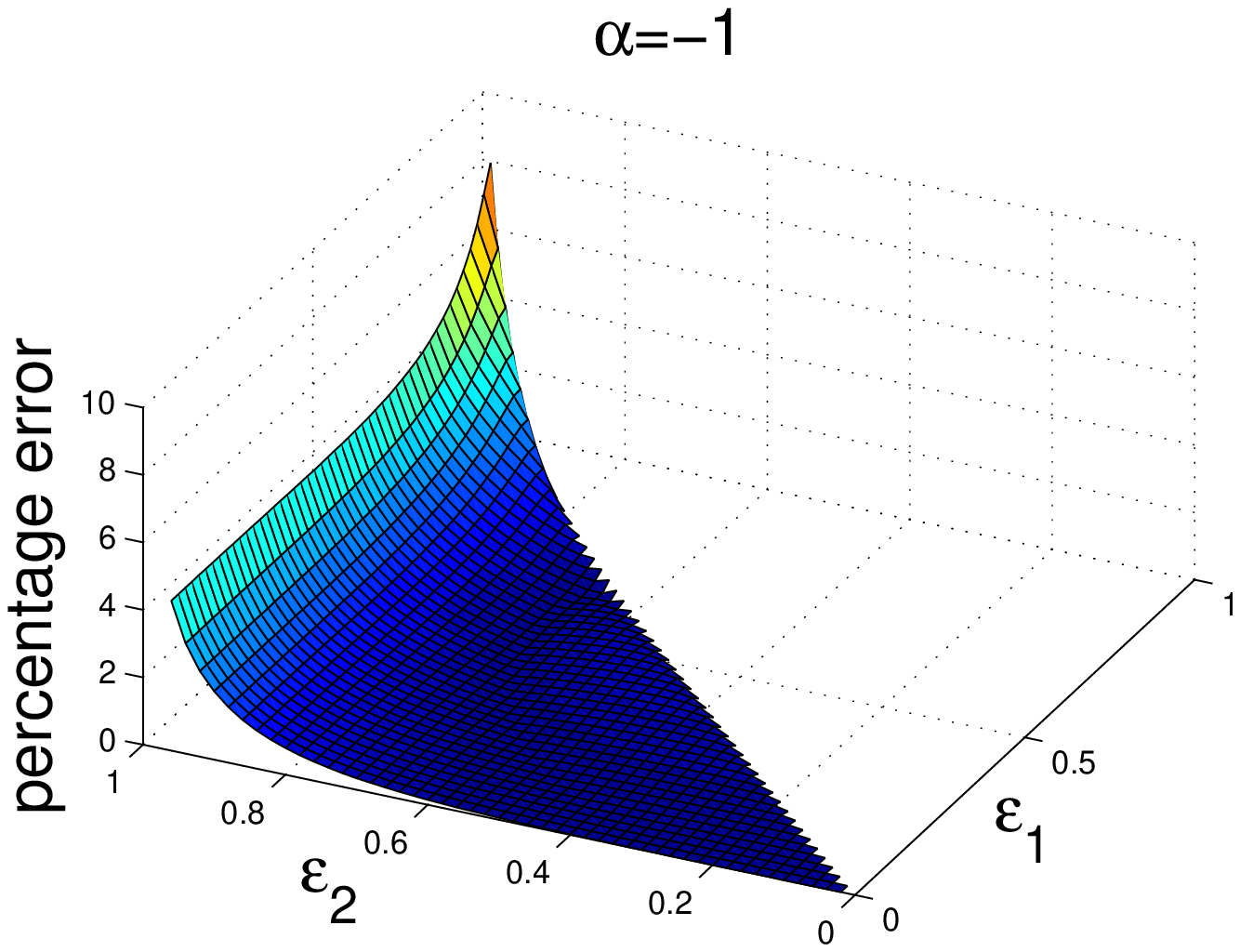}
\includegraphics[width=0.32\textwidth]{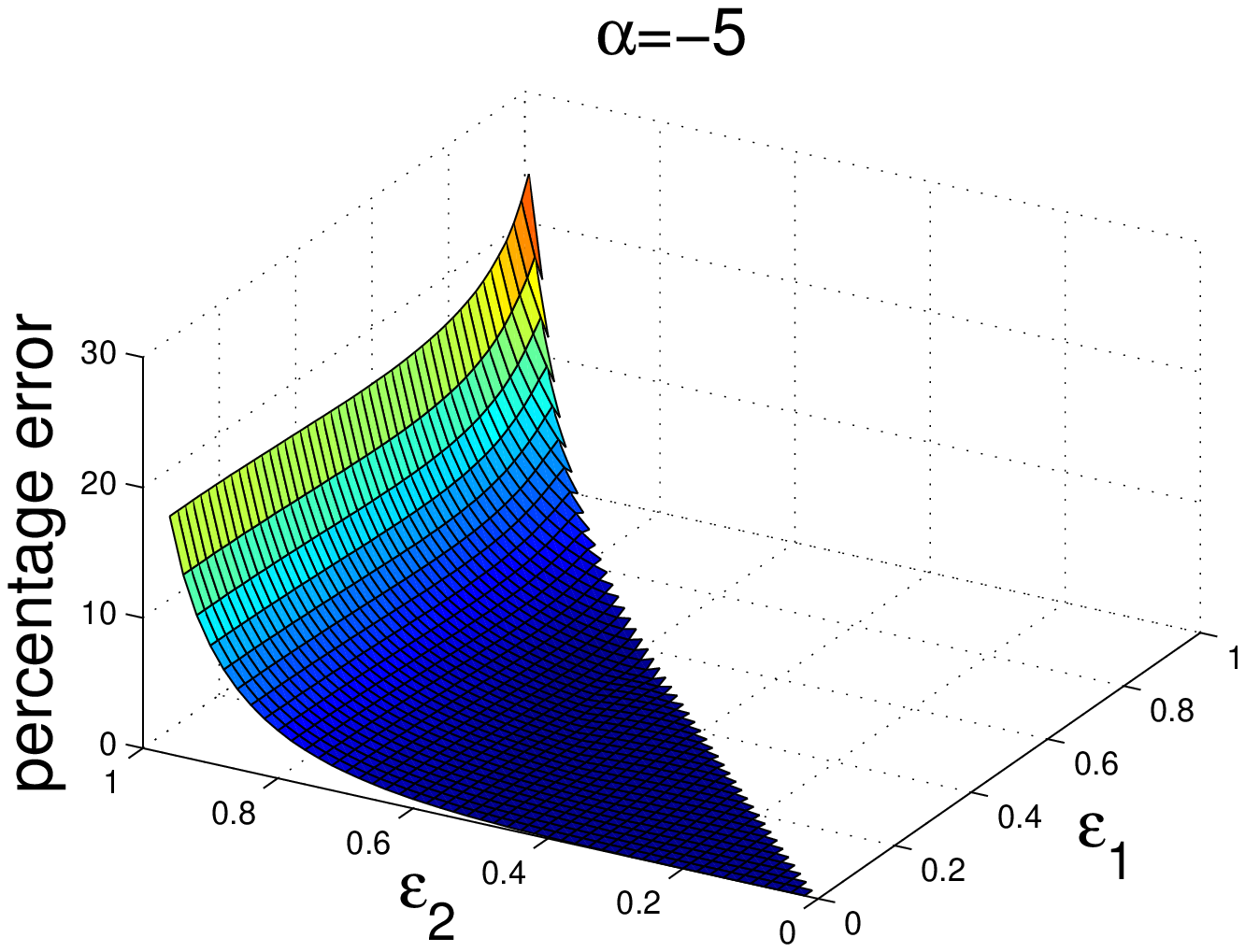}
\includegraphics[width=0.32\textwidth]{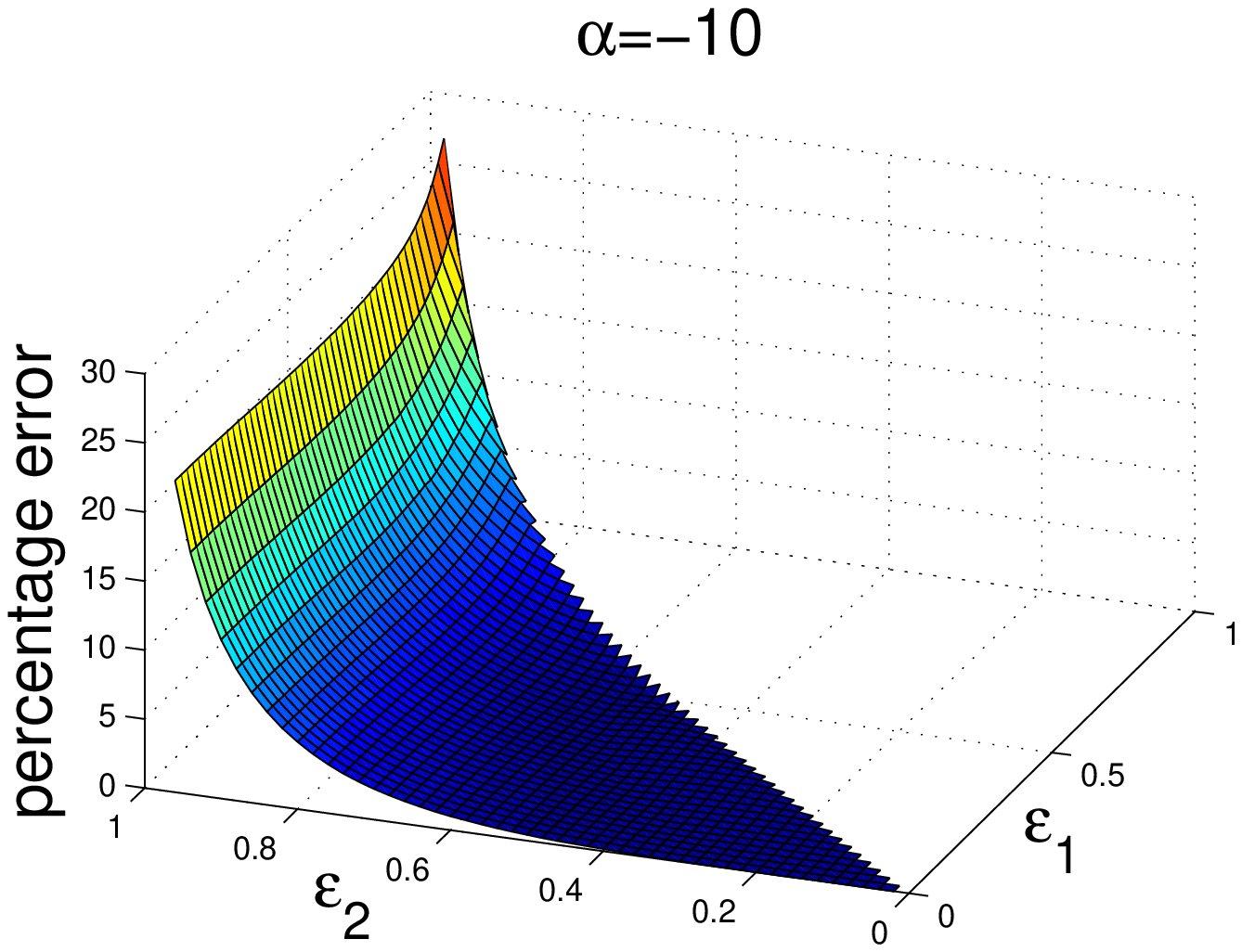}
\caption{Plot of percentage errors associated to the bound of Theorem~\ref{Thm.bound} in the class of ellipsoids.}
\label{fig:bound3d}
\end{figure}

Finally, we tested the bound of Theorem~\ref{Thm.bound} for  parallelepipeds. We will assume that each parallelepiped has length sides $l_1\leq l_2\leq l_3$ and define the
quantities $Q_1:=1-l_2/l_3$ and $Q_2 := 1-l_1/l_3$. 
Figure~\ref{fig:bound3dpar} shows
the percentage errors obtained in the class of parallelepipeds, 
as a function of $Q_1$ and $Q_2$.

\begin{figure}[ht]
\centering
\includegraphics[width=0.32\textwidth]{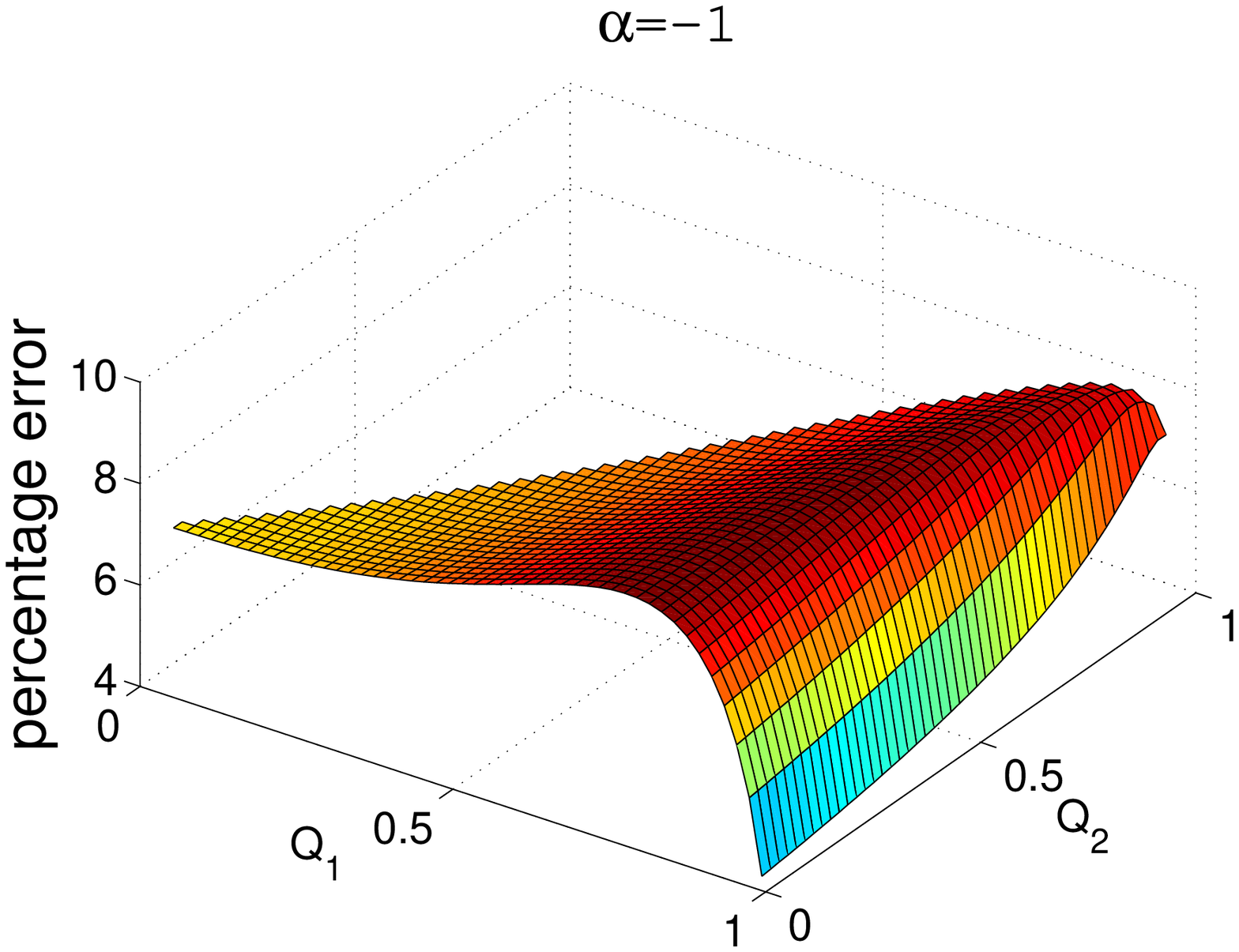}
\includegraphics[width=0.32\textwidth]{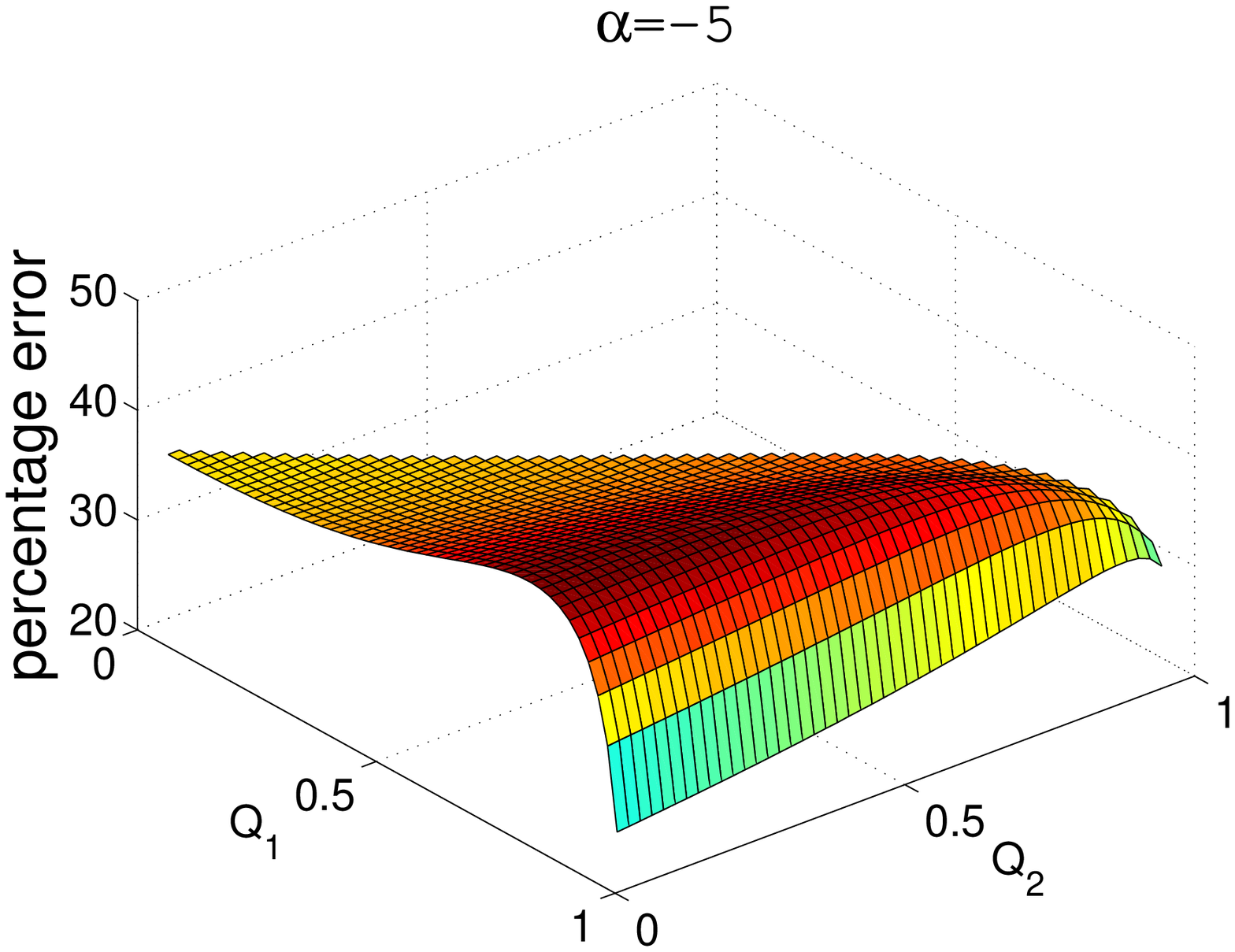}
\includegraphics[width=0.32\textwidth]{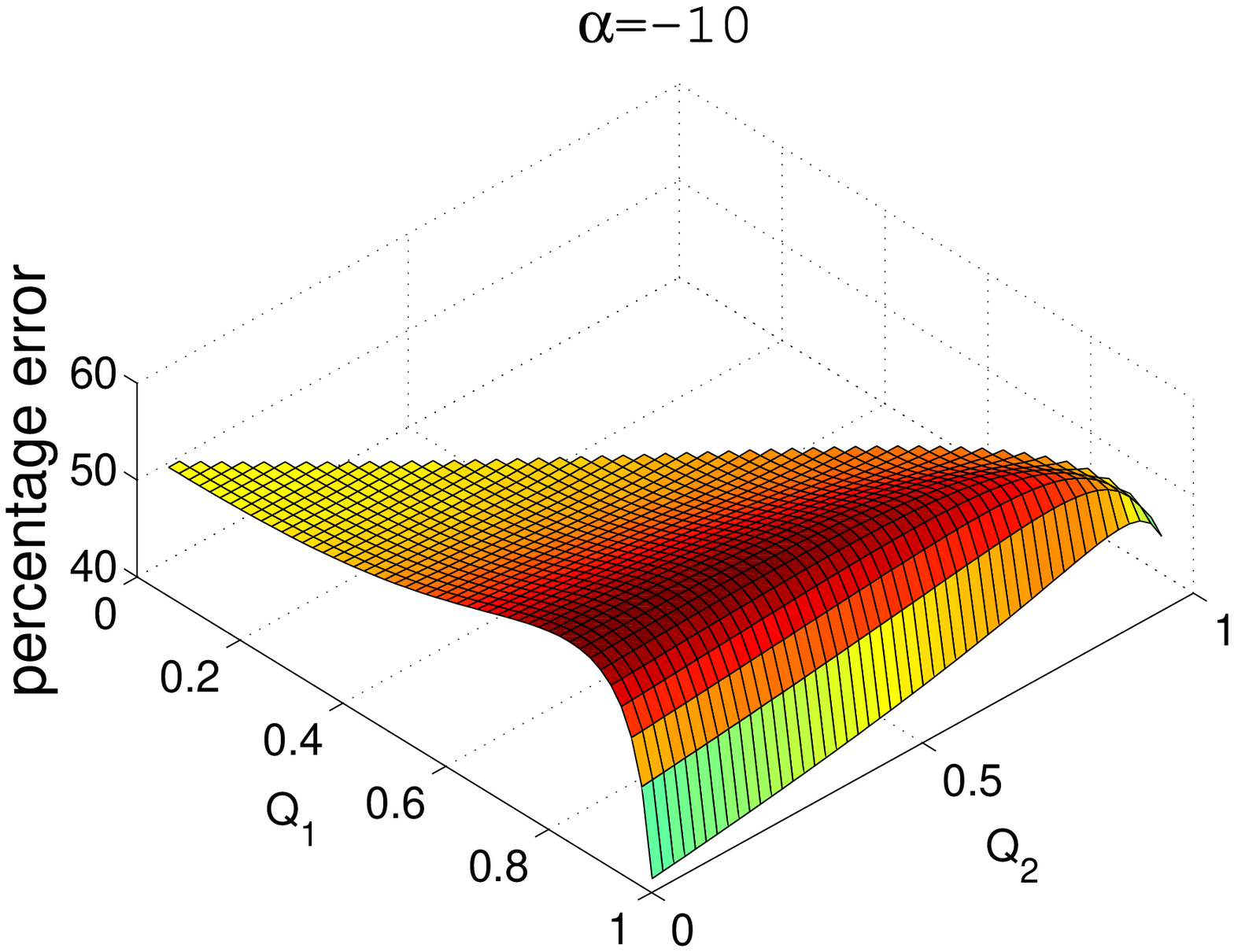}
\caption{Plot of percentage errors associated to the bound of Theorem~\ref{Thm.bound} in the class of parallelepipeds, as a function of $Q_1$ and $Q_2$.
}
\label{fig:bound3dpar}
\end{figure}
\subsection{Conjectures}
The numerical simulations carried out support the conjecture, already formulated in~\cite{FK7}, that there is a switch betwen maximisers,
with spherical shells becoming the maximisers for sufficiently large (negative) values of the parameter. We may now be more precise.

\begin{Conjecture}
 There exists a negative value of $\alpha$, say $\alpha^{*}$, such that the first eigenvalue of problem~\eqref{eq:robin} is
 maximised by the ball among all the domains with equal volume, for $\alpha\in(\alpha^{*},0)$.
 
 For $\alpha$ smaller than $\alpha^{*}$, the maximiser becomes a spherical shell whose radii increase as $\alpha$ decreases.
 
 The actual values of $\alpha^{*}$ and the radii of the shells depend on the dimension and the volume only.
\end{Conjecture}

In two dimensions, imposing the extra condition that the domain is simply connected will strongly restrict maximisers.
\begin{Conjecture}
 In two dimensions the disk maximises 
the first eigenvalue of problem~\eqref{eq:robin} for negative $\alpha$, 
among all simply connected domains with the same area.
\end{Conjecture}
Simply connectedness is is clearly not enough to restrict maximisers to balls in higher dimensions and it becomes thus
necessary to impose stronger conditions. Although there are other possibilities such as requiring that the boundary be
connected, here we just explored the case where domains are convex.
\begin{Conjecture}
 The ball maximises the first eigenvalue of problem~\eqref{eq:robin} 
for negative~$\alpha$, among all
 convex domains with the same volume.
\end{Conjecture}

Finally, our results support the conjecture that Theorem~\ref{Thm.perimeter} may be extended to any dimension.
\begin{Conjecture}
 The ball maximises the first eigenvalue of problem~\eqref{eq:robin} 
for negative $\alpha$, among all domains of equal surface area.
\end{Conjecture}

%
%

\providecommand{\bysame}{\leavevmode\hbox to3em{\hrulefill}\thinspace}
\providecommand{\MR}{\relax\ifhmode\unskip\space\fi MR }
\providecommand{\MRhref}[2]{%
  \href{http://www.ams.org/mathscinet-getitem?mr=#1}{#2}
}
\providecommand{\href}[2]{#2}

\end{document}